\documentclass[10pt]{amsart}
\usepackage[margin=1in]{geometry}
\usepackage{ytableau}

\usepackage{amssymb,amsmath,amsthm,amscd,epsf,latexsym,verbatim,graphicx,amsfonts,hyperref,epstopdf,xcolor,wasysym}
\usepackage{mathtools}

\usepackage{thmtools} 
\usepackage{thm-restate}

\usepackage{wrapfig}
\usepackage{cutwin}
\usepackage{tikz}
\usepackage{tkz-graph}
\usepackage{caption}
\usepackage{float}
\usetikzlibrary{shapes,snakes, positioning}
\usetikzlibrary{arrows}
\usetikzlibrary{shapes.misc}
\usepackage{rotating}
\usetikzlibrary{decorations.markings}

\usepackage{array}  
\usepackage{longtable} 
\usepackage{booktabs}
\newcolumntype{C}{>{$}c<{$}} 
\input{insbox}

\tikzset{cross/.style={cross out, draw=black, fill=none, minimum size=2*(#1-\pgflinewidth), inner sep=0pt, outer sep=0pt}, cross/.default={2pt}}

\theoremstyle{theorem}
\newtheorem{theorem}{Theorem}[section]

\newtheorem{lemma}[theorem]{Lemma}
\newtheorem{proposition}[theorem]{Proposition}
\newtheorem{corollary}[theorem]{Corollary}

\theoremstyle{definition}

\theoremstyle{definition}

\theoremstyle{definition}

\theoremstyle{definition}

\theoremstyle{definition}

\theoremstyle{definition}

\theoremstyle{definition}

\theoremstyle{definition}

\newtheoremstyle{named}{}{}{\itshape}{}{\bfseries}{.}{.5em}{\thmnote{#3 }#1}
\theoremstyle{named}
\newtheorem*{namedtheorem}{Theorem}

\newcommand{\ZZ}{\mathbb{Z}}
\newcommand{\RR}{\mathbb{R}}

\newcommand{\QQ}{\mathbb{Q}}
\newcommand{\CC}{\mathbb{C}}
\newcommand{\NN}{\mathbb{N}}
\newcommand{\TT}{\mathbb{T}}
\newcommand{\EE}{\mathbb{E}}

\newcommand{\Prim}{\hbox{\sf RP}}

\newcommand{\STS}{\hbox{\sf STS}}
\newcommand{\PTS}{\hbox{\sf PTS}}

\newcommand{\calH}{\mathcal{H}}
\newcommand{\Hol}{\hbox{\sf Hol}}

\newcommand{\Aut}{\hbox{\rm Aut}}

\newcommand{\lcm}{\hbox{\rm lcm}}
\newcommand{\Var}{\hbox{\rm Var}}
\newcommand{\Tr}{\hbox{\rm Tr}}

\newcommand{\Irr}{\hbox{\rm Irr}}
\newcommand{\Par}{\hbox{\rm Par}}
\newcommand{\CF}{\hbox{\rm CF}}
\newcommand{\ch}{\hbox{\rm ch}}
\newcommand{\cyc}{\hbox{\rm cyc}}

\newcommand{\cont}{\hbox{\rm cont}}

\newcommand{\SL}{\mathrm{SL}}

\DeclarePairedDelimiter{\ideal}{\langle}{\rangle}
\DeclarePairedDelimiter\floor{\lfloor}{\rfloor}
\DeclarePairedDelimiter\ceil{\lceil}{\rceil}

\begin{document}
\title{The topology and geometry of random square-tiled surfaces}

\author{Sunrose Shrestha}
\address{Department of Mathematics, Tufts University, 503 Boston Avenue, Medford, MA 02155}
\email{sunrose.shrestha@gmail.com}




\begin{abstract}A square-tiled surface (STS) is a branched cover of the standard square torus with branching over exactly one point. In this paper we consider a randomizing model for STSs and generalizations to branched covers of other simple translation surfaces which we call polygon-tiled surfaces. We obtain a local central limit theorem for the genus and subsequently obtain that the distribution of the genus is asymptotically normal. We also study holonomy vectors (Euclidean displacement vectors between cone points) on a random STS. We show that asymptotically almost surely the set of holonomy vectors of a random STS contains the set of primitive vectors of $\ZZ^2$ and with probability approaching $1/e$, these sets are equal.
\end{abstract} 
\maketitle

In this paper we will study topological and geometric statistics of a specific kind of \emph{translation surfaces} called \emph{square-tiled surfaces}. Translation surfaces form an important class of metrics on two-manifolds, namely those that admit an atlas whose 
transition functions are given by Euclidean translations.  They can be viewed from several other, equivalent perspectives:
\begin{itemize}
\item complex analysis:  a translation surface is a pair $(X,\omega)$ where $X$ is a Riemann surface and
$\omega$ is a holomorphic differential
(i.e. away from finitely many singular points, $\omega$  is the pullback of the one-form $dz$);
\item Euclidean geometry:  a translation surface is a collection of Euclidean polygons with sides glued in parallel pairs by translations.
\end{itemize}
Translation surfaces are flat everywhere except for finitely many \emph{cone} points where the angles are $2\pi (\alpha_i+1)$ for $\alpha_i \geq 1$. The number and angles of cone points is recorded as $\alpha = (\alpha_1, \dots, \alpha_s)$, and the collection of surfaces sharing the same cone point data is called a \emph{stratum}, denoted $\calH(\alpha)$. For instance, the surface in Figure \ref{fig:sqtiled} lives in $\calH(1,1)$, as it has two cone points of angle $4 \pi$ each. For any cone point data $\alpha$, the stratum $\calH(\alpha)$ has an orbifold structure with local coordinates given by \emph{holonomy vectors} (Euclidean displacement vectors between cone points.)

\begin{wrapfigure}{r}{0.4\textwidth}
\vspace{-0.2cm}
\centering
\begin{tikzpicture}
    \draw (0,0) -- (4,0); 
    \draw (0,1) -- (4,1);
    \draw (0,2) -- (2,2);
    \draw (0,0) -- (0,2);
    \draw (1,0) -- (1,2);
    \draw (2,0) -- (2,2);
    \draw (3,0) --(3,1);
    \draw (4,0) -- (4,1);
    
    \node at (2.5, 0) [below=.075cm] {$a$};
     \node at (3.5, 1) [above] {$a$};
     \node at (2.5, 1) [above] {$b$};
     \node at (3.5, 0) [below=0cm] {$b$};
     
     \foreach \i in {1,...,4}
{ \node at (\i-0.5, 0.5)  {\i};
}

    \foreach \i in {5,6}
{ \node at (\i-4.5, 1.5)  {\i};
}

     \draw [fill] (0,0) circle [radius=0.05];
     \draw [fill] (0,2) circle [radius=0.05];    
      \draw [fill] (2,0) circle [radius=0.05];
     \draw [fill] (2,2) circle [radius=0.05];   
      \draw [fill] (3,1) circle [radius=0.05];   
       \draw [fill] (4,0) circle [radius=0.05];

        \draw (0,1) node[cross=3] {};
        \draw (2,1) node[cross=3] {};
        \draw (3,0) node[cross=3] {};
        \draw (4,1) node[cross=3] {};

\end{tikzpicture}
\caption{A square-tiled surface $S \in \STS_6$: sides are glued in opposite pairs unless otherwise indicated by the letters $a$ and $b$. It is represented by the permutations $\sigma = (1,2,3,4)(5,6)$ and $\tau = (1,5)(2,6)(3,4)$ which describe the horizontal and vertical gluings respectively. Corners indicated by a dot are identified together and corners indicated by a cross are identified together. }
\label{fig:sqtiled}
\vspace{-0.1cm}
\end{wrapfigure}
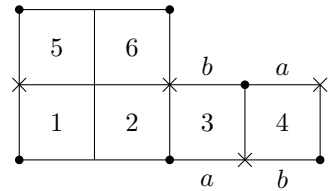

Translation surfaces have been widely studied and are useful tools in related areas of study such as Teichm\"uller dynamics, moduli problems on Riemann surfaces, and mathematical billiards. Every stratum of translation surfaces admits an action of $\SL_2(\RR)$. The orbits of translation surfaces under the diagonal subgroup of $\SL_2(\RR)$ give rise to geodesics in Teichm\"uller space under the Teichm\"uller metric. The stabilizer of the $\SL_2(\RR)$ action is called its {\em Veech group} (or group of affine homeomorphisms) and it encodes the symmetries of a translation surface.

A square-tiled surface (or STS for short) is a translation surface $(X, \omega)$ in which the Riemann surface $X$ covers the square torus $\TT=\CC/\ZZ[i]$, branched over only one point; equivalently, the Euclidean polygons can be taken to be unit squares. Let $\STS_n$ denote the collection of STSs with $n$ squares.
 
STSs are an active area of research for numerous reasons. First, they are easily described examples of celebrated translation surfaces known as \emph{Veech surfaces} which are characterized by having Veech group a lattice in $\SL_2(\RR)$. Veech surfaces are known to have very strong geometric and dynamical properties (\cite{ForTangTao,SmiWeiss,Veech3}). For example, Veech surfaces have so-called ``optimal dynamics": in any direction on the surface, each infinite trajectory is either periodic or dense. The $\SL_2(\RR)$-orbits of Veech surfaces also project to algebraic curves in the moduli space of Riemann surfaces. Secondly, their square structure means that their holonomy vectors lie in $\ZZ^2$ giving the collection of STSs a lattice-like structure in the stratum.
Under the well-known Masur-Veech volume form, STSs can be used to compute the volume of $\calH_1(\alpha)$ (the set unit area surfaces of $\calH(\alpha)$), just as integer points in $\ZZ^2$ can be used for asymptotic area calculations in $\RR^2$ (Zorich \cite{Zor2}, Eskin-Okounkov \cite{EskOk}). Thirdly, STSs can be used to tackle problems in polygonal billiards as shown by Eskin--Masur--Schmoll \cite{EskMasSchmoll} who used total counts of \emph{primitive} STSs (ones that do not cover any other STS than the square torus) in $\calH(2)$ and $\calH(1,1)$ to obtain the asymptotics for the number of closed orbits for billiards in a square table with a barrier. Fourthly, the square structure also opens up a suite of connections to analytic number theory via modular forms. For instance, Lelievre-Royer \cite{LelRoy} show that orbitwise counting functions in $\calH(2)$ are quasimodular.

We've seen complex analytical and Euclidean geometric definitions of general translation surfaces. STSs in $\STS_n$ have a third, combinatorial, description using pairs of permutations from $S_n$ (the symmetric group on $n$ letters) describing the gluings in the horizontal and vertical directions as in Figure \ref{fig:sqtiled}. In this paper we will consider $S_n \times S_n$, as a combinatorial model for $\STS_n$. This allows a study of the topological and geometric statistics. 

\subsection*{Genus of random STSs} 
Let $S(\sigma, \tau)$ be the square-tiled surfaced associated to $(\sigma, \tau) \in S_n \times S_n$. Let $G_n: S_n \times S_n \rightarrow \NN$ be the random variable defined by $G_n(\sigma, \tau) = 1 - \frac{\chi(S(\sigma, \tau))}{2}$ where $\chi(S)$ is the Euler characteristic of $S$. When $S(\sigma, \tau)$ is connected, the genus is well-defined and $G_n$ picks out the genus. Our main result is a local central limit theorem for $G_n$. Recall that for functions $f$ and $g$, $f(n) \in O(g(n))$ if there exists $M > 0$ and $N >0$ such that $|f(n)| \leq M g(n)$ for all $n > N$.
 
\begin{namedtheorem}[Genus]
\label{thm:expected}

Fix $a\in  \RR_{> 0}$. Then, uniformly (i.e. with the constant in the $O(\sqrt{\log(n)})$ term only depending on $a$) for all $\ell$  such that $n-\ell$ is even and $\frac{\ell - A(n)}{\sqrt{B(n)}} \in [-a, a]$ we have, 
$$ \Pr\left[G_n = \frac{n}{2} - \frac{\ell}{2} + 1\right] = \frac{(2+ O(\sqrt{\log(n)})\exp\left(- \frac{(\ell-A(n))^2}{2 B(n)}\right)}{\sqrt{2 \pi B(n)}}$$
where $A(n) = \sum_{i=1}^n \frac{1}{i}$ and $B(n) = \sum_{i=1}^n \frac{1}{i}-\frac{1}{i^2}$. Hence, $G_n$ is asymptotically normal with mean and variance given as:
$$\EE[G_n] = \frac{n}{2} - \frac{\log(n)}{2} - \frac{\gamma}{2} + 1 + o(1); \hspace{1cm} \Var[G_n] = \frac{\log(n)}{4}+ \frac{\gamma}{4} - \frac{\pi^2}{24} + o(1)$$
where $\gamma \approx 0.577$ is the Euler-Mascheroni constant.
\end{namedtheorem}
The distribution of the genus of square-tiled surfaces was initially obtained by S. Lechner \cite{Lech} using the representation theory of the wreath product of $S_n$ and $\ZZ/4\ZZ$. We use the representation theory of the symmetric and alternating groups and our method allows us to obtain a stronger local central limit theorem for the genus in comparison to Lechner's result. We also generalize to polygon-tiled surfaces which are branched covers of other simple translation surfaces. Moreover, our method also allows investigation of geometric features which aren't apparent from Lechner's approach. 

The Genus Theorem fits in well with the large body of already existing literature on the topology of random surfaces. Gamburd-Makover \cite{GamMak} and Pippenger-Schleich \cite{PipSchleich} studied surfaces built out of $n$ ideal hyperbolic triangles and proved that the expected genus is $n/4 - \log (n)/2 + O(1)$. Fleming-Pippenger \cite{FlemPip} generalized and studied surfaces built out of $N/k$ $k$-gons and total number of sides $N$ allowing all possible gluings assuming that $\lcm(2,k)$ divides $N$. They got large deviation results and moments of the number of vertices, a result which we will parallel in our case (Theorem \ref{thm:moments}), borrowing their proof ideas. Chmutov-Pittel \cite{ChmuPit} generalized further and studied surfaces built out of $n$ polygons with a (necessarily even) total of $N$ of sides, and computed the expected genus to be $(N/2 - n - \log(N))/2$. None of these constructions considered restricting to the case of translation surfaces. The collection of translation surfaces in their models have 0 asymptotic density and hence the general asymptotic results do not apply to translation surfaces.

It is interesting to compare the expected value in the Genus Theorem with the general case of non-translation gluings considered by Gamburd and others. For a surface (not necessarily translation) built out of $n$ squares the expected genus is $\frac{n}{2} - \frac{\log(n)}{2} - \frac{\gamma}{2} +1 - \frac{\log(4)}{2}$ which is lower than the translation case (by an additive constant of $\frac{\log(4)}{2}$). Heuristically, one of the reasons  we ``gain genus" in the translation case (apart from the possible artefact of using a different randomization model) is the fact that all surfaces constructed using only translation gluings necessarily have genus at least 1.



\subsection*{Holonomy of random surfaces} 
For any translation surface $(X,\omega)$, we can write $\Hol(X)$ for its set of holonomy vectors. In this paper we will also study holonomy vectors of STSs, all of which reside in $\ZZ^2$. For instance, the holonomy vectors of the square torus $\TT$ are given by the set
$ \Prim:= \SL_2(\ZZ) \cdot (1,0)$ (denoted $\Prim$ as it is exactly comprised of those $\ZZ^2$ vectors which have relatively prime components). Since $\TT$ does not have any cone point, we use $0$ as a marked point in order to be able to define holonomy vectors on $\TT$. Recognizing that STSs are branched covers of the standard square torus, in this paper we investigate the likelihood of an STS to have the holonomy vectors of the torus. We call a square-tiled surface $S$ a \emph{holonomy torus} if $\Hol(S) = \Prim$ and a \emph{visibility torus} if $\Hol(S) \supset \Prim$ (since the set $\Prim$ is also sometimes called the visible points of $\ZZ^2$.) Our main result on the geometry of random STSs considers the asymptotic densities of holonomy and visibility tori.

\begin{restatable*}[Holonomy]{namedtheorem}{holonomy}\label{thm:holonomy} Let $S(\sigma, \tau) \in \STS_n$  given by $(\sigma, \tau) \in S_n \times S_n$. Then, as $n \rightarrow \infty$, 
$$\Pr[S(\sigma, \tau) \text{ is a holonomy torus }] =\frac{1}{e} +O(n^{-1});\qquad \Pr[S(\sigma, \tau) \text{ is a visibility torus }] = 1 - O\left(\frac{n}{2^{n/2}}\right)$$
\end{restatable*}

There reason why the constant $1/e$ appears as the limiting proportion of holonomy tori is easily described --- (connected) holonomy tori of genus greater than one are precisely those STSs which have cone points at each corner of each square. These are in turn characterized by the pairs $(\sigma, \tau)$ for which the commutator $[\sigma, \tau]$ is a derangement (i.e. contains no fixed points). The density of such permutations turns out to be precisely 1/e.

The Holonomy Theorem is another result in the already existing sea of literature on the geometry of random surfaces. Brooks-Makover \cite{BrooksMak} use random cubic graphs with random orientation as their model for random surfaces built out of ideal hyperbolic triangles and show that almost all such surfaces have large first eigenvalues, large Cheeger constants and large embedded hyperbolic balls. Additionally, Guth-Parlier-Young \cite{GuthParYoung}, Mirzakhani-Petri \cite{MirPet}, Petri \cite{Pet1, Pet2}, Petri-Thale \cite{PetTha} and others have studied geometric properties such as length of systoles, length spectrum, pants length etc. of random (not necessarily translation) surfaces. Masur-Rafi-Randecker \cite{MasRafRan} on the other hand study the expected diameter of a translation surface in the stratum $\calH_1(2g-2)$ with respect to the Masur-Veech volume form and obtain it is bounded above by a uniform multiple of $\sqrt{\log(g)/g}$.

Saddle connections on translation surfaces have also been well-studied before. For instance, Masur \cite{Mas2, Mas3}, showed that for a fixed translation surface $S$ there are quadratic lower and upper bounds on the number of saddle connections up to length $R$ depending on $S$. Eskin-Masur \cite{EskMas} improved on that result by showing that in any fixed stratum, almost every surface (with respect to the Masur-Veech measure) has an exact quadratic asymptotic for the growth rate of saddle connections up to length $R$, with the constant solely depending on the stratum. We note that this latter result does not cover STSs since they form a measure zero set with respect to the Masur-Veech measure in any stratum. Hence, theorem \ref{thm:holonomy} can be seen as a finer, complimentary result to Eskin and Masur's theorems. 

In previous work with Wang \cite{ShresWang} we proved that for a fixed stratum, a random STS is asymptotically almost surely not a visibility torus. The Holonomy Theorem sounds quite different, but not a contradiction since in this case the stratum is not fixed. Furthermore, the Holonomy Theorem immediately implies that a random STS asymptotically almost surely has a unit holonomy vector, which is shown in \cite{ShresWang} to not be the case if we restrict to the stratum $\calH(2)$.

Our paper is organized the following way. In Section~\ref{sec:background} we give the necessary background on translation surfaces, square-tiled surfaces and the generalization to polygon-tiled surfaces. We also state some theorems and preliminary propositions pertaining to permutation statistics and representation theory of the symmetric and alternating groups. Section \ref{sec:topology} will be devoted to the proof of results on the topology of random STSs, including the Genus Theorem. Section \ref{sec:geometry} will contain the proof of the Holonomy Theorem. In Section \ref{sec:PTS} we will generalize results from Section \ref{sec:topology} to polygon-tiled surfaces. We decided to postpone the general results to the last section since the proofs for the square-tiled case essentially consist of all the key ideas of the general case. Moreover, the general proofs for polygon-tiled surfaces need to be divided into even and odd cases which makes the exposition cumbersome to read while the square-tiled case is free of this issue. Hence, the reader who is interested in simply the square-tiled case (which is the more popular case) can skip the last section entirely. Finally, we also include an Appendix that contains some basic background on combinatorial tools used. 

\subsection*{Acknowledgments} We are extremely grateful to Moon Duchin for suggesting this project and advising us through it. We also thank Bram Petri, Jayadev Athreya and Larry Guth for initial conversations and for guiding us to the existing literature on random surfaces. We would also like to express our immense gratitude to Thomas Weighill and Nate Fisher for helpful conversations throughout this project. We are also grateful to Boris Hasselblatt, Vincent Delecroix and Carlos Matheus for comments on the initial drafts.

\section{Background}
\label{sec:background}
%

\subsection{Translation surfaces and their moduli spaces}

A \textbf{translation surface} can be defined geometrically as a collection of polygons in the plane with sides identified in parallel opposite pairs by translation, up to equivalence by cut and paste operations. They can also be defined as pairs $(X, \omega)$ where $\omega$ is a holomorphic one-form on the Riemann surface $X$. Translation surfaces are locally flat except at finitely many \textbf{cone points} or \textbf{singularities} where they have a cone angle of $2 \pi \ell$ for some integer $\ell \geq 2$. A cone point of angle $2 \pi \ell$ will correspond to a zero of the one-form $\omega$ of order $\ell-1$. The group $\SL_2(\RR)$ acts on translation surfaces via its linear action on $\RR^2$ --- for $A \in \SL_2(\RR)$, and $X$ a translation surface, $A\cdot X$ is obtained by acting on the polygons of $X$ linearly. This preserves the number and order of the cone points.

A Gauss-Bonnet type theorem holds for (connected) translation surfaces: if $g$ is the genus of the surface and $\alpha_1, \ldots, \alpha_s$ are the degrees of the zeros of the one-form, then the surface must satisfy the relation $\sum_{i=1}^s \alpha_i = 2g-2.$ Genus $g$ translation surfaces then fall into finitely many \textbf{strata} $\mathcal{H}(\alpha)$ where $\alpha = (\alpha_1, \ldots, \alpha_s)$ describes the orders of the zeros of the one-form, $s$ zeros of orders $\alpha_1, \ldots, \alpha_s$. For any $g$, $\calH(2g-2)$ is called a \textbf{minimal stratum} and $\calH(\overbrace{1,\dots, 1}^{2g-2})$ is called a \textbf{principal stratum}. Although we will deal with surfaces with multiple connected components, the strata $\calH(\alpha)$ will only contain connected surfaces. 

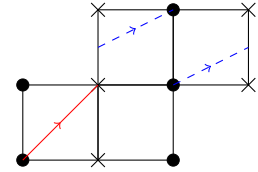
\begin{wrapfigure}{r}{0.33\textwidth}
\vspace{-0.1cm}
\centering
\begin{tikzpicture}[sq/.style=
  {shape=regular polygon, regular polygon sides=4, draw, minimum width=1.414cm}, decoration={
    markings,
    mark=at position 0.5 with {\arrow{>}}}]
\node[sq] at (0,0){};
\node[sq] at (1,0){};
\node[sq] at (1,1){};
\node[sq] at (2,1){};  


\foreach \i in {-135, 135}{
\draw[fill] (\i:1.414/2) circle [radius = 0.08];
}

\foreach \i in {45, -45}{
\draw[fill, shift = {(1,0)}] (\i:1.414/2) circle [radius = 0.08];
}

\draw[fill, shift = {(1,1)}] (45:1.414/2) circle [radius = 0.08];


\foreach \i in {-45, 45}{
\draw[shift={(2,1)}] (\i:1.414/2) node[cross=3]{};
}

\foreach \i in {45, -45}{
\draw (\i:1.414/2) node[cross=3] {} ;
}

\draw[shift = {(1,1)}] (135:1.414/2) node[cross=3]{};


\draw[red, postaction={decorate}] (-135:1.414/2) -- (45:1.414/2);

\draw[blue, dashed, shift = {(2,1)}, postaction={decorate}] (-135:1.414/2) --(0:0.5);
\draw[blue, dashed, shift = {(1,1)}, postaction={decorate}] (-180:0.5) --(45:1.414/2);

\end{tikzpicture}
\caption{Saddle connections on $S \in \calH(1,1)$ built out of unit squares. The holonomy vector of the red (solid) saddle connection is {\color{red}(1,1)} , and the holonomy vector of the blue (dashed) saddle connection is {\color{blue}(2,1)}.}
\label{fig:saddleconnection}
\vspace{-0.3cm}
\end{wrapfigure}

On a given translation surface, a {\bf saddle connection} is a straight-line local geodesic segment whose endpoints are cone points but with no cone points on the interior. If there are no cone points on a connected component of the surface, then we mark a point and look at segments from the marked point to itself as saddle connections. The {\bf holonomy} of a saddle connection is the corresponding Euclidean displacement vector. See Figure \ref{fig:saddleconnection} for some examples. The collection of holonomy vectors of a translation surface $(X, \omega)$ will be denoted $\Hol(X)$. If $X$ has multiple connected components, then $\Hol(X)$ is the union of the holonomy vectors of each of the connected components.

\subsection{Square-tiled surfaces}

Square-tiled surfaces (STSs) are translation surfaces that are branched covers of the standard square torus $\CC/\ZZ[i]$ with branching over exactly one point. Geometrically, STSs are constructed from finitely many squares with sides glued in parallel opposite pairs. The number of squares is equal to the degree of the branched cover over $\TT$. We denote the collection of STSs built out of $n$ squares as $\STS_n$. Note that $\STS_n$ also contains surfaces with multiple connected components. The natural action of $\SL_2(\RR)$ on translation surfaces restricts to an action of $\SL_2(\ZZ)$ on STSs which preserves the number of squares and the cone point data. 

Recall from introduction that an STS is a \textbf{holonomy torus} if $\Hol(S) = \Prim := \SL_2(\ZZ) \cdot (1,0)$ and a \textbf{visibility torus} if $\Hol(S) \supset \Prim$. The set $\Prim \subset \ZZ^2$ is characterized by having relatively prime components and is sometimes referred to as the primitive vectors or visible points of $\ZZ^2$. See Figure \ref{fig:holvisnonvisexamples} for holonomy and visibility tori examples and non-examples in $\calH(2)$.
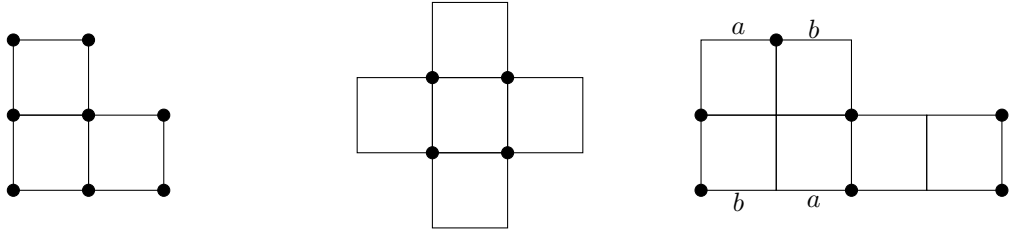
\begin{figure}[h!!]
\begin{minipage}{0.3\textwidth}
\centering
\begin{tikzpicture}
    [sq/.style=
  {shape=regular polygon, regular polygon sides=4, draw, minimum width=1.414cm}]
  \node[sq] at (0,0){};
  \node[sq] at (0,1){};
  \node[sq] at (1,0){};
%
\foreach \i in {45, 135, -45, -135}{
	\draw[fill] (\i:1.414/2) circle [radius=0.08];
	}

\foreach \i in {45, 135}	{
\draw[fill, shift ={(0,1)}] (\i:1.414/2) circle [radius=0.08];
}

\foreach \i in {45, -45}	{
\draw[fill, shift ={(1,0)}] (\i:1.414/2) circle [radius=0.08];
}

\end{tikzpicture}
\end{minipage}
\begin{minipage}{0.3\textwidth}
\centering
\begin{tikzpicture}
    [sq/.style=
  {shape=regular polygon, regular polygon sides=4, draw, minimum width=1.414cm}]
  \node[sq] at (0,0){};
  \foreach \i in {1, -1}{
  	\node[sq] at (0,\i){};
  	\node[sq] at (\i,0){};
  	}
  \foreach \i in {45, 135 , -45 ,-135}{
  \draw[fill] (\i: 1.414/2) circle [radius = 0.08];
  }
  
\end{tikzpicture}
\end{minipage}
\begin{minipage}{0.3\textwidth}
\centering
\begin{tikzpicture}
    [sq/.style=
  {shape=regular polygon, regular polygon sides=4, draw, minimum width=1.414cm}]
  \foreach \i in {1-0.5,2-0.5, -1+0.5,-2+0.5}{
  	\node[sq] at (\i,0){};
  	}
  \node[sq] at (-0.5, 1){};
  \node[sq] at (-1.5, 1){};
%
%
	\node[shift = {(-1.5, 1)}] at (90:0.65){$a$};
	\node[shift = {(-0.5, 1)}] at (90:0.65){$b$};
	\node[shift = {(-0.5, 0)}] at (-90:0.65){$a$};
	\node[shift = {(-1.5, 0)}] at (-90:0.65){$b$};

	\foreach \i in {135 ,-135}{
  \draw[fill, shift={(0.5, 0)}] (\i: 1.414/2) circle [radius = 0.08];
  }

	\foreach \i in {135 ,-135}{
  \draw[fill, shift={(-1.5, 0)}] (\i: 1.414/2) circle [radius = 0.08];
  }
  
  \foreach \i in {135 ,-135}{
  \draw[fill, shift={(2.5, 0)}] (\i: 1.414/2) circle [radius = 0.08];
  }
  
  \draw[fill, shift={(-1.5, 1)}] (45: 1.414/2) circle [radius = 0.08];

\end{tikzpicture}
\end{minipage}

\caption{Examples in $\calH(2)$. Unless otherwise indicated, sides are glued in opposite pairs. \emph{Left}: A holonomy torus in $\calH(2)$. Using Proposition \ref{prop:althol}, we see it is indeed a holonomy torus. \emph{Center}: A visibility torus that is not a holonomy torus in $\calH(2)$. Since it contains $(2,0)$ in its holonomy set, it is not a holonomy torus, but using Proposition \ref{prop:altvis}, we can verify it is a visibility torus. \emph{Right}: A non-visibility STS in $\calH(2)$. Checking visually, we see that $(1,0)$ is not in its holonomy set.}
\label{fig:holvisnonvisexamples}

\end{figure}

Note that since $\Hol(S)$ of a surface $S$ with multiple connected components is the union of the holonomy sets of the components, if $S$ has a square torus connected component, then the holonomy set of the square torus component is defined to be $\Prim$ (by marking a point), and consequently $S$ is a visibility torus. If in addition, the holonomy set of each component is $\Prim$, then $S$ is a holonomy torus. 
The following proposition, first proved with Wang for connected surfaces, then gives an alternate characterization of holonomy tori which will be useful.
\begin{proposition}[Shrestha-Wang \cite{ShresWang}]\label{prop:althol}$S \in \STS_n$ is a holonomy torus if and only if every corner of every square is either a singularity or a non-singular point in a square torus component.
\end{proposition} 
\begin{proof} If $S$ is not a holonomy torus, then it has a vector not in $\Prim$ which is the holonomy vector of a saddle connection. This saddle connection then must go through a corner of a square that is neither a singularity nor a marked point in a $\TT$ component. Conversely, if there exists a corner of a square that is not a singularity (nor a marked point), then extending a line segment through this point in any rational slope direction yields a non-primitive vector as a saddle connection and hence the surface is not a holonomy torus.
\end{proof} 
With Wang, we also obtained a characterization of visibility tori in terms of the $\SL_2(\ZZ)$ action on $\STS_n$.
\begin{proposition}[Shrestha-Wang \cite{ShresWang}]\label{prop:altvis}$S \in \STS_n$ is a visibility torus if and only if every surface in the $\SL_2(\ZZ)$ orbit of $S$ has a unit length horizontal saddle connection.
\end{proposition}

We next state a theorem proved with Wang which asserts that the visibility properties of STSs in a fixed stratum are governed by the number of squares. We will use part (3) of this theorem in our proof of the Holonomy Theorem.

\begin{theorem}[Shrestha-Wang \cite{ShresWang}]\label{thm:shresthawang}

For a fixed stratum $\calH(\alpha) = \calH(\alpha_1, \dots, \alpha_s)$, the visibility properties of STSs in the stratum are governed by their number of squares, $n$, in the following way, illustrated in the figure below:

\begin{enumerate}
\item There are no STSs in $\calH(\alpha)$ with fewer than $2g+s-2$ squares. 
\item An STS $S$ has $n=2g+s-2$ squares if and only if $\Hol(S) = \Prim$ for all $S \in \calH(\alpha) \cap \STS_n$.
\item If $2g+s-2 < n \leq 4g+2s-5$, then all $S \in \calH(\alpha) \cap \STS_n$ are visibility tori; in fact $\Prim \subsetneq \Hol(S)$.
\item If $n=4g+2s-4$, then there exists $S \in \calH(\alpha) \cap \STS_n$ that is not a visibility torus.
\item There exists $N(\alpha)$ such that if $n > N(\alpha)$ then all $S \in \calH(\alpha) \cap \STS_n$ are non-visibility tori, i.e. for all such STSs, there exists $v \in \Prim \smallsetminus \Hol(S)$.
\end{enumerate}

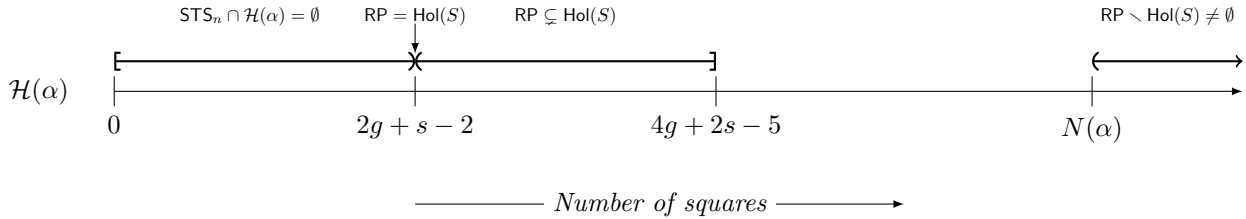
\begin{figure}[h!!!!]
\begin{tikzpicture}
\node at (-3,0){$\calH(\alpha)$};
\draw[-latex]   (-2,0) -- (13,0) ; 

\draw[shift={(-2,0)},color=black] (0,0.2) -- (0,-0.2) node[below] 
{$0$};



\draw[{[-)}, thick] (-2,.4) -- (2, .4);
\draw[{(-]}, thick] (2,.4) -- (6, .4);
\draw[{(- >}, thick] (11,.4) -- (13, .4);

\node[scale=0.7] at (-0.2,1){$\STS_n \cap \calH(\alpha) = \emptyset$};
\node[shift={(2,0)},scale=0.7]at (0, 1){$\Prim  = \Hol(S)$};
\node[scale=0.7] at (4,1){$\Prim \subsetneq  \Hol(S)$};
\node[scale=0.7] at (12,1){$\Prim \smallsetminus \Hol(S) \neq \emptyset$};

\draw[shift={(2,0)},color=black, -latex] (0,0.9)--(0,0.5);
\node[scale=0.7] at (8.35,0.8){};
\draw[shift={(-1,0)},-latex] (3,-1.5) -- (9.5,-1.5) node[midway,fill=white] {Number of squares};

\draw[shift={(2,0)},color=black] (0,0.2) -- (0,-0.2) node[below] 
{$2g+s-2$} ;
\draw[shift={(6,0)},color=black] (0, 0.2) -- (0,-0.2) node[below] 
{$4g+2s-5$} ;
\draw[shift={(11,0)},color=black] (0,0.2) -- (0,-0.2) node[below] 
{$N(\alpha)$} ;











\end{tikzpicture}
\label{fig:mconstants}
\end{figure}
\end{theorem}

\subsection{Surfaces tiled by polygons with $2k$ sides} One can easily generalize the construction of square-tiled surfaces from squares to other even polygons. For a fixed integer $k \geq 2$, one can build a surface by taking finitely many regular unit area $2k$-gons and gluing the sides in parallel opposite pairs. Such a surface is called a \textbf{polygon-tiled surface}. These surfaces were initially studied in \cite{Aul} as examples of platonic solids with a choice of horizontal direction. We denote by $\PTS_{n,2k}$ the collection of polygon-tiled surfaces built out of $n$ $2k$-gons. Note that $\STS_n = \PTS_{n, 4}$. Figures \ref{fig:hexocttiled} show examples for $k=3, 4$. Also note that when considering the polygons, we fix (only up to translation) an embedding in the plane so that two of the sides of each polygon are vertical. 

 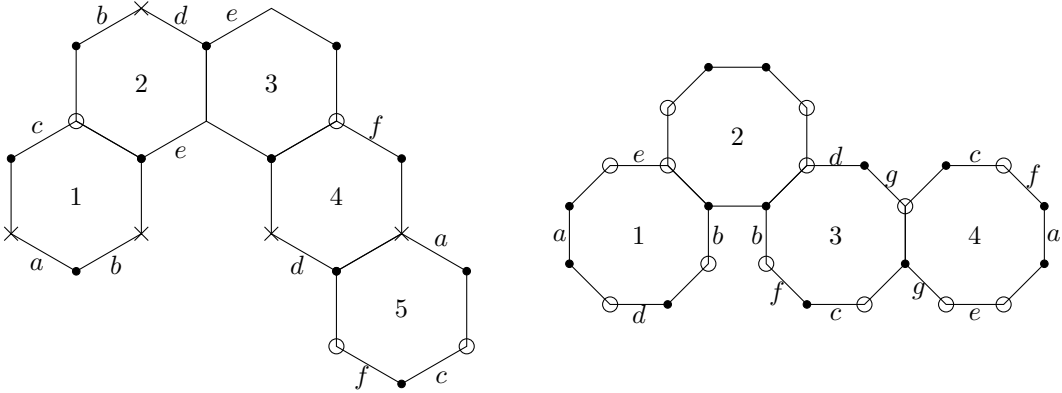
\begin{figure}[h!!]
 \begin{minipage}{0.45\textwidth}
 \centering
 \begin{tikzpicture}[hexagon/.style=
  {shape=regular polygon, regular polygon sides=6, draw, minimum width=2cm, shape border rotate=90}, rec/.style= {shape=regular polygon, regular polygon sides=4, draw, minimum width=0.1cm}]
  \node[hexagon] at (0,0) (1) {3};
  \node[hexagon] at (-180:1.73){2};
  \node[hexagon] at (-60:1.73){4};
  \node[hexagon] at (-60:2*1.73){5};
  \node[hexagon][shift={(-180:1.73)}] at (-120:1.73){1};
  
\begin{scope}[shift = {(-180:1.73)}]
 \begin{scope}[shift = {(-120:1.73)}]
  \foreach \i in {30,150,-90}{
  \draw [fill] (\i:1) circle [radius=0.05];
  }
  \node at (-60:1.05){$b$};
  \node at (-120:1.05){$a$};
  \node at (120:1.05){$c$};
  
  \foreach \i in {-30,-150}{
  \draw (\i:1) node[cross=3] {};
  
  }
  
  \end{scope}
 \end{scope}

  \begin{scope}[shift = {(-180:1.73)}]
  \foreach \i in {30,150,-90}{
  \draw [fill] (\i:1) circle [radius=0.05];
  } 
  \node at (-60:1.05){$e$};
  \node at (60:1.05){$d$};
  \node at (120:1.05){$b$};
  \foreach \i in {-150}{
  \draw (\i:1) circle [radius=0.1];
  }
  \draw (90:1) node[cross=3] {};

 \end{scope}

  \foreach \i in {30,150,-90}{
  \draw [fill] (\i:1) circle [radius=0.05];
  
  }
  
  \node at (120:1.05){$e$};
  \foreach \i in {-30}{
  \draw (\i:1) circle [radius=0.1];
  }

  
  \begin{scope}[shift = {(-60:1.73)}]
  \foreach \i in {30,150,-90}{
  \draw [fill] (\i:1) circle [radius=0.05];
  }
  \node at (-120:1.05){$d$};
  \node at (60:1.05){$f$};
  \foreach \i in {-30,-150}{
  \draw (\i:1) node[cross=3] {};
  
  }

 \end{scope}

 
 \begin{scope}[shift = {(-60:2*1.73)}]
  \foreach \i in {30,150,-90}{
  \draw [fill] (\i:1) circle [radius=0.05];
  }
  
  \node at (-60:1.05){$c$};
  \node at (60:1.05){$a$};
  
  \node at (-120:1.05){$f$};
  
  \foreach \i in {-30,-150}{
  \draw (\i:1) circle [radius=0.1];
  }
 \end{scope}
  \end{tikzpicture}
  \end{minipage}
  \begin{minipage}{0.45\textwidth}
  \centering
\begin{tikzpicture}
    [oct/.style=
  {shape=regular polygon, regular polygon sides=8, draw, minimum width=2cm, shape border rotate=90}]
  \node[oct] at (0,0) (2) {2};
  \node[oct] at (-2*22.5:1.85) {3};
  \node[oct] at (-6*22.5:1.85) {1};
  \node[oct][shift = {(-2*22.5:1.85)}] at (0: 1.85){4};
  
  \foreach \i in {3,5,-3,-5}{
  \draw [fill] (\i*22.5:1) circle [radius=0.05];
  }
  
  \foreach \i in {1,7,-1,-7}{
  \draw (\i*22.5:1) circle [radius=0.1];
  }
  
  
  \begin{scope}[shift = {(-6*22.5: 1.85)}]
  \foreach \i in {-3,7,-7}{
  \draw [fill] (\i*22.5:1) circle [radius=0.05];
  }
  \foreach \i in {-1,-5,5}{
  \draw (\i*22.5:1) circle [radius=0.1];
  }

	\node at (0:1.05) {$b$};
	\node at (180:1.05){$a$};
	\node at (90:1.05){$e$};
	\node at (-90: 1.05){$d$};

  \end{scope}
  
  
  \begin{scope}[shift = {(-2*22.5: 1.85)}]
  \foreach \i in {3,-1,-5}{
  \draw [fill] (\i*22.5:1) circle [radius=0.05];
  }
  \foreach \i in {1,-3,-7}{
  \draw (\i*22.5:1) circle [radius=0.1];
  }

  \node at (45:1.05) {$g$};
	\node at (180:1.05){$b$};
	\node at (-90:1.05){$c$};
	\node at (90: 1.05){$d$};
  \node at (-6*22.5:1.1){$f$};
  \end{scope}
  
  
  \begin{scope}[shift = {(-2*22.5: 1.85)}]
  \begin{scope}[shift = {(0*22.5:1.85)}]
  \foreach \i in {1,-1,5}{
  \draw [fill] (\i*22.5:1) circle [radius=0.05];
  }
  \foreach \i in {-3,3,-5}{
  \draw (\i*22.5:1) circle [radius=0.1];
  }
  \node at (-6*22.5:1.05) {$g$};
	\node at (0:1.05){$a$};
	\node at (-90:1.05){$e$};
	\node at (90: 1.05){$c$};
  \node at (45:1.1){$f$};
  \end{scope}
  \end{scope}

%
%
%
%
%
%
	
%
%
%

  \end{tikzpicture}  
  \end{minipage}
  
\caption{Examples of polygon-tiled surfaces. Unless otherwise indicated, sides are glued in opposite pairs. Left: A surface in $\PTS_{5,6}$ given by permutations $\sigma_1 = (1)(2,3)(4)(5)$, $\sigma_2 = (1,5,4,3,2)$ and $\sigma_3 = (1,5,4,2)(3)$. It belongs to the stratum $\calH(4,1,1)$. Right: A surface in $\PTS_{4,8}$ given by permutations $\sigma_1 = (1,3,4)(2)$, $\sigma_2 = (1)(2,3)(4)$, $\sigma_3 = (1,3,4)(2)$, $\sigma_4 = (1,2)(3,4)$. It belongs to stratum $\calH(5,5)$. The procedure to associate permutations to encode the gluings is described generally in Section \ref{sec:randmodelGEN}.}
\label{fig:hexocttiled}
\end{figure}

\subsection{Randomizing model for $\STS_n$}\label{sec:randmodel}

 Given a square-tiled surface $S \in \STS_n$, first label the squares $\{1, \dots, n\}$. We can then read off a pair of permutations $\sigma, \tau \in S_n$ which describe the horizontal and vertical gluings respectively --- $\sigma$ and $\tau$ are defined so that for any square labeled $i$, the square labeled $\sigma(i)$ has its left side glued to the right side of $i$ and $\tau(i)$ has its bottom side glued to the top of $i$. For example, from the square-tiled surface in Figure \ref{fig:sqtiled} we obtain $\sigma = (1,2,3,4)(5,6)$ and $\tau = (1,5)(2,6)(3,4)$. Relabeling the surface corresponds to simultaneous conjugation of the permutation pair by the relabeling permutation.
 
Conversely, given a pair of permutations, $(\sigma, \tau) \in S_n \times S_n$, one can associate a labeled square-tiled surface $S(\sigma, \tau) \in \STS_n$. This surface is connected if and only if the subgroup generated by the permutations $\ideal{\sigma,\tau} \subset S_n$, is a transitive subgroup. By a classical theorem of Dixon \cite{Dix}, the probability that a random pair  of permutations in $S_n$ generates a transitive subgroup is at least $1 - 1/n + O(n^{-2})$. Hence, a uniformly distributed pair of permutations in $S_n$ gives a connected surface asymptotically almost surely as $n \rightarrow \infty$. So, $S_n \times S_n$ with the uniform measure  is a combinatorial model for random STSs built out of $n$ squares.

\subsection{Genus of STSs from the randomizing model} To determine the genus of a random STS using the random model, we start with the Euler characteristic formula: 
$2-2g = V -E + F$. For a surface with $n$ squares, $F = n$ and $E=2n$ and hence the genus is given by,
$ g = \frac{n}{2}-\frac{V}{2} + 1$. So, the problem of determining the genus boils down to computing the number of equivalence classes of vertices under the gluing. Note that a vertex need not be a cone point. The following proposition from folklore gives a combinatorial condition allowing us to count the number of such equivalence classes simply from the combinatorial information.

\begin{proposition}[Combinatorial condition for identifying vertices]\label{prop:vertexgluing} Let $S\in \STS_n$ be given by $(\sigma, \tau) \in S_n \times S_n$. Let $c = \sigma\tau\sigma^{-1}\tau^{-1}$ be the commutator of $\sigma$ and $\tau$. Then, the bottom left corner of square $i$ is identified with the bottom left corner of square $j$ if and only if $i$ and $j$ are in the same cycle of $c$. Hence, the number of equivalence classes of bottom left corners is equal to the number of cycles of $c$.
\end{proposition}
\begin{proof}
%
%
First label the edges of each of the squares as $e_1, \dots, e_4$ so that $e_1$ is the right vertical edge,  $e_2$ is the top edge, $e_3$  the left edge, and $e_4$ the bottom one.

($\Leftarrow$) Assume $i$ and $j$ are in the same cycle of $c$. If $i = j$ there is nothing to show. So assume $i \neq j$ and that there exists $\ell \geq 1$ such that $c^\ell (i) = j$.  We then proceed by induction on $\ell$. We start with the base case $\ell = 1$ so that $c(i) = \sigma\tau \sigma^{-1} \tau^{-1}(i) = j$. See Figure \ref{fig:vertexfromcommutatorSQR} for an illustration of the proof. By definition of $\sigma$ and $\tau$, 
\begin{align*} \text{edge } e_4 \text{ of } i \sim \text{edge }e_2 \text { of } \tau^{-1}(i) &\implies \text{bottom left corner of } i \sim \text{top left corner of }\tau^{-1}(i)\\
\text{edge } e_3 \text{ of } \tau^{-1}(i) \sim \text{edge } e_1 \text{ of } \sigma^{-1}\tau^{-1}(i) &\implies \text{top left corner of } \tau^{-1}(i) \sim \text{top right corner of }\sigma^{-1}\tau^{-1}(i)\\
\text{edge } e_2 \text{ of } \sigma^{-1}\tau^{-1}(i) \sim \text{edge } e_4 \text{ of } \tau\sigma^{-1}\tau^{-1}(i) &\implies \text{top right corner of } \sigma^{-1}\tau^{-1}(i) \\ &\hspace{1cm}\sim \text{bottom right corner of }\tau\sigma^{-1}\tau^{-1}(i)\\
\text{edge } e_1 \text{ of } \tau\sigma^{-1}\tau^{-1}(i) \sim \text{edge } e_3 \text{ of } c(i) = j &\implies \text{bottom right corner of } \tau\sigma^{-1}\tau^{-1}(i) \\ &\hspace{1cm}\sim \text{bottom left corner of }c(i)=j
\end{align*}
\begin{figure}[h!]
\begin{tikzpicture}[sq/.style=
  {shape=regular polygon, regular polygon sides=4, draw, minimum width=2*1.414cm}]
	\node[sq] at (0,0){$i$}; 
	
\foreach \edge [count = \i] in {$e_1$, $e_2$, $e_3$, $e_4$}{
\node at (\i*90-90:0.8cm){\edge};
}
	\node[sq] at (-90:2cm){$\tau^{-1}(i)$};

\node[shift = {(-90:2cm)}] at (90:0.8cm){$e_2$};
\node[shift = {(-90:2cm)}] at (180:0.8cm){$e_3$};

	\node[sq, shift = {(-90:2)}] at (180:2cm){};
	\node[shift = {(-90:2)}][scale = 0.8] at (180:2cm){$\sigma^{-1}\tau^{-1}(i)$};
	\begin{scope}[shift = {(-90:2)}]
		\begin{scope}[shift ={(180:2)}] 
	 \draw [shift={(0.05,0)}] (90:0.85cm) -- (90:1.15cm);
  \draw [shift={(-0.05,0)}] (90:0.85cm) -- (90:1.15cm);	
  \end{scope}
  \end{scope}
	
	\begin{scope}[shift = {(5, 0)}]
	\node[sq] at (0,0){};
	\node[scale=0.8] at (0,0){$\tau\sigma^{-1}\tau^{-1}(i)$};
	\draw [shift={(0.05,0)}] (-90:0.85cm) -- (-90:1.15cm);
  \draw [shift={(-0.05,0)}] (-90:0.85cm) -- (-90:1.15cm);	
	\node[sq] at (0:2){$c(i)=j$};
	 \begin{scope}
	\clip (1,-1) rectangle (-1, 1);
	
	\draw (-45:1.414) circle [radius = 0.5];

	\end{scope}
	\draw[fill] (-45:1.414) circle [radius = 0.1];

	\end{scope}
	
	\begin{scope}
	\clip (-1,-1) rectangle (1.414, 1.414);
	
	\draw (-135:1.414) circle [radius = 0.5];
	
	\end{scope}
	
	\begin{scope}
	\clip (-1,-1) rectangle (1, -3);
	
	\draw (-135:1.414) circle [radius = 0.5];
	
	\end{scope}
	
	\begin{scope}
	\clip (-1,-1) rectangle (-3, -3);
	
	\draw (-135:1.414) circle [radius = 0.5];
	
	\end{scope}
	
	\draw[fill] (-135:1.414) circle [radius = 0.1];		
	\end{tikzpicture}
\caption{Proof Illustration: Action of the commutator $c$ corresponds to going around a vertex clockwise starting at the bottom left corner of $i$ and ending at the bottom left corner of $c(i) = j$.}
\label{fig:vertexfromcommutatorSQR}
\end{figure}
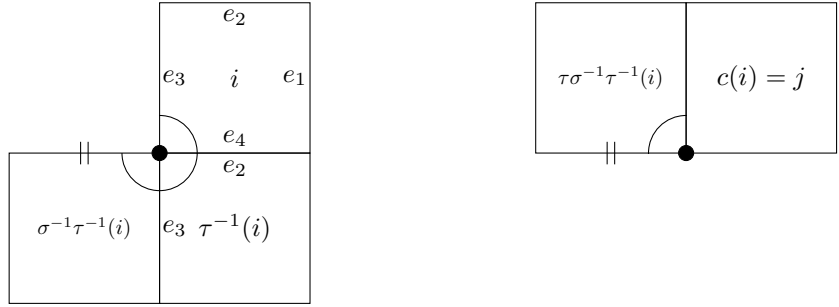
Now assume as induction hypothesis that $c^\ell(i) = j$ implies that $i$ and $j$ have their bottom left corners identified. Then, if $c^{\ell+l}(i) = c^\ell(c(i))=j$, by induction hypothesis, $c(i)$ and $j$ have their bottom left corners identified. By the base case, $i$ and $c(i)$ have their bottom left corners identified which implies the same for $i$ and $j$.

($\Rightarrow$) Now assume that $i$ and $j$ have their bottom left corners identified. Call the vertex represented by the bottom left corner of $i$ to be $v$. Let $2 \pi \ell$ be the angle around $v$. We will proceed via induction on $\ell$. 
When $\ell = 1$, $v$ is not a cone point, and is made up of the bottom left of $i$, the top left of $\tau^{-1}(i)$, the top right of $\sigma^{-1}\tau^{-1}(i)$ and the bottom right of $\tau\sigma^{-1}\tau^{-1}(i)$ each of which contribute $\pi/2$ angle. This implies that $i = j$, and hence $i$ and $j$ are in the same cycle of $c$ vacuously.
As induction hypothesis, assume that for any STS, $S(\sigma', \tau')$, whenever $i$ and $j$ have their bottom left corners identified and the angle at the vertex represented is $2\pi \ell$ then $i$ and $j$ are in the same cycle of the commutator $[\sigma', \tau']$. 
Now, assume the angle around $v$ is $2 \pi (\ell+1)$. Let $(i, c(i), \dots, c^r(i))$ be the cycle of $c$ containing $i$. Note that by definition, edge $e_1$ of $\tau\sigma^{-1}\tau^{-1}(i)$ is identified with edge $e_3$ of $c(i)$ and edge $e_1$ of $\sigma^{-1}(i)$ is identified with edge $e_3$ of $i$. Then define a new surface $S'$ from $S$ with all the same edge gluings except swapping these two edge identifications --- now edge $e_1$ of $\tau\sigma^{-1}\tau^{-1}(i)$ is identified to edge $e_3$ of $i$ and edge $e_1$ of $\sigma^{-1}(i)$ is identified to edge $e_3$ of $c(i)$. This new surface has a new commutator $c'$ such that $c'(i) = i$ and $(c(i), \dots, c^r(i))$ is a cycle. But on $S'$, $c(i)$ and $j$ have their bottom left corners identified, and the angle around the vertex represented by these corners is $2 \pi \ell$. Hence, by induction hypothesis, $c(i)$ and $j$ are in the same cycle of $c'$. But, that means $c^m(i) = j$ for some $1\leq m\leq r$.\end{proof}

Proposition \ref{prop:vertexgluing} then allows us to conclude the following regarding the angle around the vertices. 

\begin{proposition}[Commutator cycle type determines cone point data]\label{prop:TopFromCommutator} Let $S \in \STS_n$ be given by $(\sigma, \tau) \in S_n \times S_n$. If the cycle type of $c := [\sigma, \tau]$ is $(1^{\xi_1}, 2^{\xi_2}, \dots, n^{\xi_n})$, then $S$ has $\xi_\ell$ vertices with angle $2\pi\ell$ for each $\ell = 1, \dots, n$.
\end{proposition}
\begin{proof}Consider a cycle of $c$ and let $i$ be in this cycle. From Proposition \ref{prop:vertexgluing}, the bottom left corner of $i$ represents the vertex corresponding to this cycle. Furthermore, as illustrated in figure \ref{fig:vertexfromcommutatorSQR}, following the action of $c$ from $i$ to $c(i)$ corresponds to traversing angle $2\pi$ around this vertex. Hence, if the cycle containing $i$ is of length $\ell$, then the angle around the corresponding vertex is $2\pi\ell$.
\end{proof}

Note that Proposition \ref{prop:TopFromCommutator} allows us to determine the stratum of the surface simply from the combinatorics of the permutations describing it. For instance, if the commutator of the two permutations describing a connected square-tiled surface is a product of $m$ disjoint transpositions, then the stratum is $\calH(\overbrace{1,\dots, 1}^m)$. See Figure \ref{fig:topfromcommexamples} for more illustrations of Proposition \ref{prop:TopFromCommutator}.

\begin{figure}[h!]
\begin{minipage}{0.45\textwidth}
\flushleft
\begin{tikzpicture}[sq/.style=
  {shape=regular polygon, regular polygon sides=4, draw, minimum width=1.414cm}]
  \node[sq] at (0,0){$1$};
  \node[sq] at (1,0){$2$};
  \node[sq] at (2,0){$3$}; 
  \node[sq] at (3,0){$4$};
  \node[sq] at (3,1){$5$};
  \node[sq] at (4,1){$6$};
  
  \node at (1, 0.65) {$a$};
  \node at (3, -0.65) {$a$};
  
   \node at (2, 0.65) {$b$};
  \node at (1, -0.65) {$b$};

   \node at (3, 1+0.65) {$c$};
  \node at (4, 1-0.65) {$c$};

   \node at (4, 1+0.65) {$d$};
  \node at (2, -0.65) {$d$};
  
  \foreach \i in {1, 2, 3, 4}{
  
  \draw[fill] (\i*90-45:1.414/2) circle [radius =0.05]; 
  }
  
   \foreach \i in {1, 3, 4}{
  
  \draw[fill, shift={(3,0)}] (\i*90-45:1.414/2) circle [radius =0.05]; 
  }
  \draw[fill, shift={(4,1)}] (45:1.414/2) circle [radius =0.05]; 
  \draw[fill, shift={(3,1)}] (90+45:1.414/2) circle [radius =0.05]; 
  \draw[fill, shift={(2,0)}] (90+45:1.414/2) circle [radius =0.05];

  \draw[shift ={(2,0)}] (45:1.414/2) node[cross=3]{}; 
   \draw[shift ={(2,0)}] (-135:1.414/2) node[cross=3]{}; 
   \draw[shift ={(4,1)}] (135:1.414/2) node[cross=3]{};
    \draw[shift ={(4,1)}] (-45:1.414/2) node[cross=3]{};

\end{tikzpicture}
\end{minipage}
\begin{minipage}{0.45\textwidth}
\flushright
\begin{tikzpicture}[sq/.style=
  {shape=regular polygon, regular polygon sides=4, draw, minimum width=1.414cm}]
  \node[sq] at (0,1){$1$};
  \node[sq] at (1,1){$2$};
  \node[sq] at (1,0){$3$}; 
  \node[sq] at (2,0){$4$};
  \node[sq] at (3,0){$5$};
  \node[sq] at (4,0){$6$};
  \node[sq] at (5,0){$7$};
  \node[sq] at (6,0){$8$};

  
  \node at (0, 1+0.65){$a$};
  \node at (1,-0.65){$a$};
  
  \node at (1, 1+0.65){$b$};
  \node at (4,-0.65){$b$};
  
  \node at (2, 0.65){$c$};
  \node at (5,-0.65){$c$};

  \node at (3, 0.65){$d$};
  \node at (6,-0.65){$d$};
  
  \node at (4, 0.65){$e$};
  \node at (0,1-0.65){$e$};
  
  \node at (5, 0.65){$f$};
  \node at (3,-0.65){$f$};
  
  \node at (6, 0.65){$g$};
  \node at (2,-0.65){$g$};'
  

  \draw[fill, shift={(0,1)}] (135:1.414/2) circle [radius = 0.05];
  \draw[fill, shift={(0,1)}] (-135:1.414/2) circle [radius = 0.05];
  
  \draw[fill, shift={(1,1)}] (45:1.414/2) circle [radius = 0.05];
  
  \draw[fill, shift={(1,1)}] (-45:1.414/2) circle [radius = 0.05];
  
  \draw[fill, shift={(1,0)}] (-135:1.414/2) circle [radius = 0.05];
  
  \draw[fill, shift={(3,0)}] (45:1.414/2) circle [radius = 0.05];
  
  \draw[fill, shift={(4,0)}] (-45:1.414/2) circle [radius = 0.05];
  
  \draw[fill, shift={(6,0)}] (-45:1.414/2) circle [radius = 0.05];

  
  \draw[shift ={(1,1)}] (-135:1.414/2) node[cross=3]{}; 
  \draw[shift ={(2,0)}] (-45:1.414/2) node[cross=3]{}; 
  \draw[shift ={(4,0)}] (45:1.414/2) node[cross=3]{};  
  \draw[shift ={(6,0)}] (45:1.414/2) node[cross=3]{}; 
  
  
  \draw[shift={(6,0)}] (135:1.414/2) circle [radius = 0.08];
  
  \draw[shift={(1,0)}] (-45:1.414/2) circle [radius = 0.08];
  \draw[shift={(1,1)}] (135:1.414/2) circle [radius = 0.08];
  \draw[shift={(4,0)}] (-135:1.414/2) circle [radius = 0.08];

\end{tikzpicture}
\end{minipage}
\caption{Illustrating Proposition \ref{prop:TopFromCommutator}. Left:
A square-tiled surface in $\STS_6$ given by permutations $\sigma = (1, 2, 3, 4)(5,6)$ and $\tau = (1)(2,4,5,6,3)$. The commutator is $[\sigma, \tau] =(1,6,4,2)(3,5)$, indicating that the surface is in $\calH(3,1)$. Right: A square-tiled surface in $\STS_8$ given by permutations $\sigma = (1,2)(3,4,5,6,7,8)$ and $\tau = (1,3,2,6)(4,7,5,8)$. The commutator is $[\sigma, \tau] =(1,3,7)(2,5)(4,6)(8)$ indicating that the surface is in $\calH(2,1,1)$. } 
\label{fig:topfromcommexamples}
\end{figure}
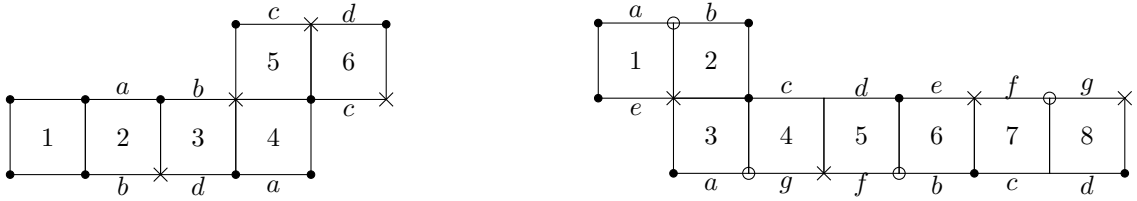

Since the commutator captures information regarding the cone point data of an STS, consider the commutator map $w_c:  S_n \times S_n \rightarrow S_n$ so that $w_c(\sigma, \tau):=[\sigma, \tau]$. Given a random STS, $(\sigma, \tau) \in S_n\times S_n$, using Proposition \ref{prop:vertexgluing} we see that the number of vertices $V$ equals the number of cycles of the permutation  $w_c(\sigma, \tau)$. Hence, we define the random variable
\begin{equation}\label{eq:numcycles}C_{w_c, n} = \# \text{ cycles of }w_c(\sigma,\tau)
\end{equation}
where $(\sigma, \tau) \in S_n \times S_n$ is uniformly distributed. 

With that we can then define $G_n: S_n \times S_n \rightarrow \NN$ to be the random variable
$$ G_n = \frac{n}{2} - \frac{C_{w_c, n}}{2} + 1 $$
which outputs $1 - \chi/2$ where $\chi$ is the Euler characteristic. Note that $G_n$ is exactly the genus when the surface is connected. 

 The map $w_c: S_n \times S_n \rightarrow S_n$ induces a probability measure $P_{w_c, n}$ on $S_n$ such that 
$ P_{w_c, n} (\pi) = \frac{\# w_c^{-1}(\pi)}{(n!)^2}.$ Frobenius \cite{Frob} showed that the probability measure $P_{w_c,n}$ takes the special form,
\begin{equation}\label{eqn:probspecialform}P_{w_c,n}(\pi) = \frac{1}{n!}\sum_{\chi \in \Irr(S_n)}\frac{\chi(\pi)}{\chi(1)}
\end{equation}
As the image of the map $w_c: S_n \times S_n \rightarrow S_n$ is $A_n$ \cite{Ore}, the probability measure $P_{w_c,n}$ is supported on $A_n$.


\subsection{Holonomy from the randomizing model}
The random model also gives information on the geometry of STSs. In particular, using Proposition \ref{prop:althol} and Proposition \ref{prop:TopFromCommutator} we can get a combinatorial characterization of holonomy tori.
\begin{proposition}\label{prop:HolFromComb}Let $S:=S(\sigma, \tau)$ be a square-tiled surface. Then $S$ is a holonomy torus if and only if either $[\sigma, \tau]$ is a derangement or each of its fixed points is also fixed by $\sigma$ and $\tau$.
\end{proposition}
\begin{proof}
Assume $S$ is a holonomy torus and that $[\sigma, \tau]$ is not a derangement. Let $i$ be a fixed point of $[\sigma, \tau]$. As a consequence of Propositions \ref{prop:vertexgluing} and \ref{prop:TopFromCommutator}, the bottom left corner of square $i$ is not a singularity. By Proposition \ref{prop:althol}, the bottom left corner of $i$ is a point in a square torus component. The only way this can happen is if the square $i$ itself belongs to a square torus component. But a square torus component contains only one square. Hence, $i$ should have its opposite sides glued in pairs, implying $\sigma(i) = i$ and $\tau(i) = i$.

For the converse, first note that every corner of every square is either itself a bottom left corner or is glued to the bottom left corner of some square. If $[\sigma, \tau]$ is a derangement, then by Propositions \ref{prop:vertexgluing} and  \ref{prop:TopFromCommutator} we know every bottom left corner of every square is a singularity which implies that $S(\sigma, \tau)$ is a holonomy torus by Proposition \ref{prop:althol}.
In case $[\sigma, \tau]$ has fixed points, by Proposition \ref{prop:TopFromCommutator}, we know the fixed points are in one-to-one correspondence with nonsingular vertices of $S$ (i.e. points of $S$ arising from gluing exactly 4 square corners). Assuming then each fixed point $i$ of $[\sigma, \tau]$ is also a fixed point of $\sigma$ and $\tau$, we know that $i$ forms a $\TT$ component, and the bottom left corner of $i$ is surely in this component. Hence, every corner of every square is either a singularity or a point in a $\TT$ component. By Proposition \ref{prop:althol}, $S$ is a holonomy torus. \end{proof}

\subsection{Permutation statistics}We will need some results on statistics of permutations as well. We begin with the classically known asymptotic density (as $n\rightarrow \infty$) of permutations in $S_n$ without fixed points.

\begin{proposition}\label{prop:derangementAn}
The probability that a uniformly random permutation in $A_n$ (or $S_n)$ does not have a fixed point tends to $1/e$ as $n \rightarrow \infty$ whereas the expected number of fixed points is 1. 
\end{proposition}
The probability of a random permutation in $S_n$ not having a fixed point is easy to compute. For any $k \in \{1, \dots, n\}$, the proportion of permutations moving it is exactly $1-1/n$. So, the proportion of permutations moving all of the $\{1, \dots, n\}$ is $(1-1/n)^n$. This converges to $1/e$ as $n \rightarrow \infty$.  

Similarly,

\begin{proposition}\label{prop:transpositions}The probability that a uniformly distributed random permutation from $A_n$ is a product of disjoint transpositions (and 1-cycles) tends to 0 as $n \rightarrow \infty$, and the probability that it is a cyclic permutation also tends to $0$ as $n \rightarrow \infty$. 
\end{proposition}
\begin{proof}
Given a cycle type $\lambda = (1^{\xi_1}, 2^{\xi_2}, \dots, n^{\xi_n})$ of permutations with $\xi_i$ cycles of length $i$, it is classically known that the probability that a uniformly distributed permutation in $S_n$ has cycle type $\lambda$ is given by,
$$ \Pr[\pi \in S_n \text{ has cycle type } \lambda] =  \prod_{j=1}^n \left(\frac{1}{j}\right)^{\xi_j}\cdot \frac{1}{\xi_j!}$$
If $\lambda$ is a valid cycle type for an even permutation, then 
$$\Pr [\eta \in A_n \text{ has cycle type } \lambda ] = 2 \prod_{j=1}^n \left(\frac{1}{j}\right)^{\xi_j}\cdot \frac{1}{\xi_j!}.$$
Hence,
$$\Pr[\eta \in A_n \text{ is a cycle}] = 2\sum_{\substack{k=0 \\ k \not\equiv n \mod 2}}^{n-1} \frac{1}{k!} \cdot \left(\frac{1}{n-k}\right)$$
which is obtained by summing over the cycle types $(1^k,2^0, \dots, (n-k)^1, \dots, n^0)$ with $k$ fixed points and an $(n-k)$-cycle. Clearly, this series is bounded above by
$$ 2 \cdot \sum_{k=0}^{n-1} \frac{1}{k!}\frac{1}{n-k}.$$
According to the OEIS \cite{OEIS}, the series $n!\sum_{k=0}^{n-1} \frac{1}{k!}\frac{1}{n-k} \sim e\cdot(n-1)!$. This implies 
$$\Pr[\eta \in A_n \text{ is a cyclic permutation}] \rightarrow 0.$$
Similarly, 
$$\Pr[\eta \in A_n \text{ is a product of disjoint transpositions}] = 2\sum_{\substack{k=0 \\ k \equiv n \mod 4}}^{n-4} \frac{1}{k!} \cdot \frac{1}{\left(\frac{n-k}{2}\right)!} \cdot \left(\frac{1}{2}\right)^{\frac{n-k}{2}}$$
which is obtained by summing over the cycle types $(1^k, 2^{(n-k)/2}, 3^0, \dots, n^0)$ with $k$ fixed points and $(n-k)/2$ disjoint transpositions. Note that when $k \equiv n \mod 4$ and $k \leq n-4$, $(n-k)/2 \geq 2$ so that
$$ \left(\frac{n-k}{2}\right)!\cdot 2^{(n-k)/2} \geq 4\cdot\left(\frac{n-k}{2}\right)! \geq 4\cdot\frac{n-k}{2} \geq n-k.$$
Hence,
$$\Pr[\eta \in A_n \text{ is a product of disjoint transpositions}] \leq 2 \cdot \sum_{k=0}^{n-1} \frac{1}{k!}\frac{1}{n-k} \rightarrow 0$$
as $n \rightarrow \infty$. 
\end{proof}

Next, define the random variables  
$$C_{\eta, n} = \# \text{ cycles of }\eta; \qquad C_{\pi, n} = \# \text{ cycles of }\pi$$
for $\eta \in A_n$ and $\pi \in S_n$ uniformly distributed. We will use the following result by Fleming-Pippenger which gives large deviation statistics for $C_{\eta, n}$.

\begin{lemma}[Fleming-Pippenger \cite{FlemPip}]\label{lem:largedevunif}
Let $\eta$ be uniformly distributed on $A_n$ for $n \geq 3$. Then, 
$$\Pr(C_{\eta, n} \geq t) = O\left(\frac{n}{2^t}\right)$$
\end{lemma}
 We will also utilize the following classically known statistics of $C_{\eta, n}$ and $C_{\pi, n}$ (which can be found in \cite{FlemPip} in more generality).
\begin{proposition}\label{prop:expAnSn} For $X = C_{\eta, n}$ or $C_{\pi, n}$ defined above, we have
$$\EE(X) = \log(n) + \gamma + o(1); \qquad 
\Var(X) = \log(n) + \gamma - \frac{\pi^2}{6} + o(1)$$
where $\gamma \approx 0.577$ is the Euler-Mascheroni constant.
\end{proposition}

\subsection{Tools from Representation Theory}

For our proofs we will use a few common tools from the representation theory of finite groups. 

Let $G$ be a finite group and $P$ and $Q$ be probability measures on $G$. Define the \textbf{total variation distance} between $P$ and $Q$ to be $||P-Q|| = \max_{A \subset G}|P(A) - Q(A)|$. Also, define the \textbf{Fourier transform}  of $P$ at a representation $\rho$ to be
$$\hat{P}(\rho) = \sum_{g \in G} P(g) \rho(g)$$
We use the same definition for the Fourier transform of any function $f$ on $G$. 
We will utilize the following lemma, initially proved for a general group by Diaconis-Shahshahani and carefully applied to the case of $A_n$ by Chmutov-Pittel, which gives an upper bound on the total variation distance between $P_{w_c, n}$, the measure on $A_n$ induced by the commutator word map $w_c$, and $U_{A_n}$, the uniform measure on $A_n$, in terms of the Fourier transform of $P_{w_c, n}$:
\begin{lemma}[Diaconis-Shahshahani \cite{DiaShah} Chmutov-Pittel \cite{ChmuPit}]\label{lem:DSCPupperbound} Let $U_{A_n}$ and $P_{w_c, n}$ be defined as above. Then for $n \geq 5$.
$$||P_{w_c, n} - U_{A_n}||^2 \leq \frac{1}{4} \sum_{\lambda \neq (n),\lambda \neq (1^n)}\dim \rho^\lambda \Tr[\hat{P}_{w_c, n}(\rho^\lambda) \hat{P}_{w_c, n}(\rho^\lambda)^*]$$
\end{lemma}
If the probability measure in consideration is a class function, its Fourier transform at a representation takes a particularly nice form:
\begin{lemma}[Diaconis-Shahshahani \cite{DiaShah}]\label{lem:DSfourier}Let $G$ be a finite group and $P$ a probability measure on $G$ constant on conjugacy classes. Let $\rho$ be a linear representation of $G$. Then, the Fourier transform of $P$ at the representation $\rho$, can be expressed as,
$$\hat{P}(\rho) = \frac{1}{\dim (\rho)} \sum_{K \text{ conjugacy class of }G} P[K]\cdot|K| \cdot \chi^\rho(K) \cdot I_{\dim(\rho)}$$ 
where $P[K]$ is the value that $P$ takes on a single element of $K$, $\chi^\rho$ is the character associated to the representation $\rho$ and $I_{\dim(\rho)}$ is the $\dim(\rho) \times \dim(\rho)$ identity matrix.
\end{lemma}

We will also use the following estimate by Liebeck-Shalev:

\begin{lemma}[Liebeck-Shalev \cite{LieSha}]\label{lem:liebeckshalev} Let $s \in \RR_{> 0}$ be fixed. Then,
$$\sum_{\chi \in \Irr(S_n)} \chi(1)^{-s} = 2 +  O(n^{-s}).$$
\end{lemma}
The sum in this lemma is also sometimes referred to as the Witten zeta function over $S_n$.

\section{Proof of Topology results}\label{sec:topology}
In this section we will first prove a result that allows us to estimate moments of the number of vertices on random STSs by moments of the number of cycles of a uniformly random permutation from $A_n$.  We will then prove the Genus Theorem. We begin with a few lemmas, the first of which will show that $P_{w_c, n}$ converges to the uniform distribution on $A_n$ asymptotically:

\begin{lemma}\label{lem:distconv} Let $U_{A_n}$ and $P_{w_c, n}$ be as previously defined. Then,
$$||P_{w_c, n} - U_{A_n}|| = O\left(\frac{1}{n}\right)$$

\end{lemma}

\begin{proof}
We start with the Diaconis-Shahshahani, Chmutov-Pittel upper bound from Lemma \ref{lem:DSCPupperbound}:
$$||P_{w_c, n} - U_{A_n}||^2 \leq \frac{1}{4} \sum_{\lambda \neq (n), \lambda \neq (1^n)} \dim (\rho^\lambda) \Tr[\hat{P}_{w_c, n}(\rho^\lambda)\hat{P}_{w_c, n}(\rho^\lambda)^*]$$ for $n \geq 5$. Using equation \ref{eqn:probspecialform} and Lemma \ref{lem:DSfourier}, we get 
\begin{align*}
\hat{P}_{w_c,n}(\rho^\lambda) &= \frac{1}{\dim \rho^\lambda} \sum_{K \text{ conjugacy classes}}P_{w_c,n}(K)\cdot |K| \cdot \chi^\lambda(K) I_{\dim \rho^\lambda}\\
&= \frac{1}{\dim \rho^\lambda}  \sum_{K \text{ conjugacy classes}}\left[ \left(\frac{1}{n!}\sum_{\omega \vdash n} \frac{\chi^\omega(K)}{\dim \rho^\omega}\right) |K| \chi^\lambda(K)\right] \cdot I_{\dim \rho^\lambda}\\
&= \frac{1}{\dim \rho^\lambda}\sum_{\omega \vdash n}\frac{1}{\dim \rho^\omega}\cdot \frac{1}{n!}\sum_{K \text{ conjugacy classes}}\chi^\omega(K)\cdot |K| \cdot \chi^\lambda(K) \cdot I_{\dim \rho^\lambda}\\
&=  \frac{1}{\dim \rho^\lambda}\sum_{\omega \vdash n}\frac{1}{\dim \rho^\omega}\cdot \frac{1}{n!}\sum_{\sigma \in S_n}\chi^\omega(\sigma)\cdot\chi^\lambda(\sigma) \cdot I_{\dim \rho^\lambda}\\
&= \frac{1}{(\dim \rho^\lambda)^2} \cdot  I_{\dim \rho^\lambda} \hspace{4.5cm} \text{ (due to orthogonality of irreducible characters)}
\end{align*}

Since this is a real diagonal matrix, $\hat{P}_{w_c,n}(\rho^\lambda)^* = \hat{P}_{w_c, n}(\rho^\lambda)$. So,
$$||P_{w_c, n} - U_{A_n}||^2 \leq \frac{1}{4} \sum_{\lambda \neq (n), \lambda \neq (1^n)} \frac{1}{(\dim \rho^\lambda)^2}$$

Note that $\chi^\lambda(1) = \dim \rho^\lambda$ and $\dim \rho^\lambda =1 $ when $\lambda = (n) \text{ or } (1^n)$. Then using Lemma \ref{lem:liebeckshalev}, we obtain the required estimate.
\end{proof}
We will next state the following exercise from Stanley's Enumerative Combinatorics II \cite{Stan2}, as a Lemma we will use later. A general version of this lemma, along with the proof, needed for general polygon-tiled surfaces appears in Section \ref{sec:PTS} as Lemma \ref{lem:genfcnGEN}.

\begin{lemma}\label{lem:genfcn} The generating function for $C_{w_c, n}$ is given by,
$$ G(q):= \sum_{\sigma, \tau \in S_n} q^{C_{w_c, n}(\sigma, \tau)} = (n!) \sum_{\lambda \vdash n}\prod_{r \in \lambda}(q+ \cont(r))$$
where 
\begin{itemize}
\item the sum is over unordered partitions $\lambda$ of $n$,
\item the product is over squares in the young diagram of a partition,
\item $\cont(r)$ is the content of square $r$.
\end{itemize}
\end{lemma}
We refer the reader to Appendix \ref{sec:combinatorics} for the definitions of the combinatorial notions in this lemma.

Using Lemma \ref{lem:genfcn} we can now obtain a large deviation result for $C_{w_c,n}$.

\begin{lemma}[Large deviations] \label{lem:largedev}Let $C_{w_c,n}$ be as before. Then,
$\Pr(C_{w_c,n} \geq t) = O\left(\frac{n}{2^t}\right)$
\end{lemma}

\begin{proof}Let $g(q) := \frac{G(q)}{(n!)^2}$ be the probability generating function of $C_{w_c, n}$. By Lemma \ref{lem:genfcn} we have,
$$g(q) = \frac{1}{(n!)}\sum_{\lambda \vdash n} \prod_{r \in \lambda}(q+\cont(r))$$
Then for $q \geq 1$, 
\begin{align*}
\Pr(C_{w_c, n} \geq t]) = \sum_{s \geq t} \Pr(C_{w_c, n} = s)
\leq \frac{1}{q^t} \sum_{s \geq t} \Pr(C_{w_c, n} = s)q^s 
\leq \frac{g(q)}{q^t}
\end{align*}

Taking $q = 2$, we see that in the expression for $g(2)$, the only partitions that survive in the sum over the partitions are the ones that do not have any square with content -2. These are exactly the partitions of the form $(n-j, j)$ for $j = 0$ to $\floor{n/2}$. Hence, 


$$\Pr(C_{w_c,n} \geq t) \leq 
\frac{1}{2^tn!} \sum_{j=0}^{\floor{n/2}} \prod_{r \in (n-j, j)}(2 + \cont(r))= \frac{1}{2^t}\sum_{j=0}^{\floor{n/2}}\frac{(n-j+1)!j!}{n!}\leq \frac{1}{2^t}\left(n+1 + \sum_{j=1}^{\floor{n/2}} 1\right)= O\left(\frac{n}{2^{t}}\right)$$
\end{proof}

\subsection{Moments of vertex count}
We now state and prove the following theorem which allows us to approximate the moments of number of vertices by the moments of number of cycles in random permutations of $A_n$. The proof technique is inspired by the proof of Theorem 1.3 of Fleming-Pippenger \cite{FlemPip}. 

\begin{theorem}[Moments of vertex count]\label{thm:moments}

Let $p(x)$ be a polynomial of degree $\ell$ that is non-negative and non-decreasing for $x \geq 0$. Then,
$$\EE[p(C_{w_c,n})] = \EE[p(C_{\eta, n})] + O\left(\frac{\log^\ell(n)}{n}\right)$$

\end{theorem} 
\begin{proof} 
First by definition,
$$\EE[p(C_{w_c, n}] = \sum_{s=0}^n p(s)\Pr(C_{w_c, n} = s)$$
Rewriting this in terms of the tail probabilities, we have,
$$\EE[p(C_{w_c, n})] = p(0) + \sum_{s=1}^n (p(s) - p(s-1))\Pr(C_{w_c, n}\geq s)$$
Now, subtracting the analogous expression for $\EE[p(C_{\eta, n})]$ and taking absolute values, 
\begin{align*}|\EE[p(C_{w_c, n})] - \EE[p(C_{\eta, n})]| &\leq\left(\sum_{1\leq s \leq t} (p(s) - p(s-1))\right)|\Pr(C_{w_c, n}\geq s) - \Pr(C_{\eta, n} \geq s)|\\
&+  \sum_{t < s \leq n} ((p(s) - p(s-1) \Pr(C_{w_c, n} \geq s)) + \sum_{t < s \leq n} ((p(s) - p(s-1) \Pr(C_{\eta, n} \geq s))\\
&\leq ||P_{w_c, n} - U_{A_n}||\cdot p(t) + \Pr(C_{w_c, n} \geq t)\cdot p(n) + \Pr(C_{\eta, n} \geq t)\cdot p(n)\\
\end{align*}
The first term is estimated by applying Lemma \ref{lem:distconv}, the second term is approximated by applying Lemma \ref{lem:largedev} and the third term is approximated using Lemma \ref{lem:largedevunif} and so we obtain,
$$|\EE[p(C_{w_, n})] - \EE[p(C_{\eta, n})]| \leq O\left(\frac{t^\ell}{n}\right) + O\left(\frac{n^{\ell+1}}{2^t}\right) + O\left(\frac{n^{\ell+1}}{2^t}\right)$$
Taking $t = \ceil{\frac{(\ell+1)\log(n)}{\log(2)}}$ we obtain the required estimate. \end{proof}


\subsection{Genus Theorem}
The expected value and variance of $G_n$ stated in the Genus Theorem is immediate from Theorem \ref{thm:moments} and Proposition \ref{prop:expAnSn}.
\begin{corollary}\label{cor:expected}
The expected value and variance of the genus random variable $G_n$ are given as follows:
$$\EE[G_n] = \frac{n}{2} - \frac{\log(n)}{2} - \frac{\gamma}{2} + 1 + o(1); \hspace{1cm} \Var[G_n] = \frac{\log(n)}{4}+ \frac{\gamma}{4} - \frac{\pi^2}{24} + o(1)$$
\end{corollary}
%
%
%
Additionally we obtain a local central limit theorem for the genus distribution. In \cite{ChmuPit},  Chmutov-Pittel point out the consequences of their Theorem 2.1. They string together results from  Sachkov \cite{Sach}, Kolchin \cite{Kol}, Menon \cite{Men}, Bender \cite{Ben}, Canfield \cite{Can}, and Lebowitz et al.\cite{Leb} and observe that
$$\Pr[C_{\eta} = \ell] = \frac{(2+ O(\Var[C_{\pi, n}]^{-1/2}))\exp\left(-\frac{(\ell-\EE[C_{\pi, n}])^2}{2\Var[C_{\pi, n}]}\right) }{\sqrt{2\pi\Var[C_{\pi, n}]}} \hspace{0.5cm} \text{uniformly for } \ell \text{ such that } \frac{\ell-\EE[C_{\pi, n}]}{\sqrt{\Var[C_{\pi, n}]}} \in [-a,a]$$
for fixed $a > 0$. Using this they obtain local central limit theorems for both the distribution of number of vertices and the genus. These consequences follow solely from the total variation convergence of the relevant probability measure in their case to the uniform measure on the alternating group and do not depend on the specific probability measure. Hence, the consequences also apply to our Lemma \ref{lem:distconv}, and so we obtain a local central limit theorem for $G_n$.

\begin{theorem}\label{thm:lclt}
Fix $a\in  \RR_{> 0}$. Then, uniformly for all $\ell$ such that $n-\ell$ is even and $\frac{\ell - \EE[C_{\pi, n}]}{\sqrt{\Var[C_{\pi, n}]}} \in [-a, a]$ we have, 
$$ \Pr\left[G_n = \frac{n}{2} - \frac{\ell}{2} + 1\right] = \frac{(2+ O(\sqrt{\log(n)})\exp\left(- \frac{(\ell-\EE[C_{\pi, n}])^2}{2 \Var[C_{\pi, n}]}\right)}{\sqrt{2 \pi \Var[C_{\pi, n}]}}$$


\end{theorem}

Theorem \ref{thm:lclt} together with Corollary \ref{cor:expected} is precisely the Genus Theorem.

\subsection{Number of cone points} Let $s_n: S_n \times S_n \rightarrow \NN$ be the random variable that counts the number of cone points of a random surface $(\sigma, \tau) \in S_n \times S_n$. As a Corollary to Theorem \ref{thm:moments}, we also get the expected value of $s_n$.

\begin{corollary}\label{cor:conepoints} The expected number of cone points on a random STS is given by,
$$\EE[s_n] = \log(n) + \gamma -1 + o(1)$$
\end{corollary}
\begin{proof}
Let $S(\sigma, \tau)$ be a square-tiled surface. Then, from Proposition \ref{prop:TopFromCommutator}, we know that the number of cone points (i.e. vertices with angle greater than  $2\pi$) is given by $\sum_{\ell = 2}^n \xi_\ell$ when the cycle type of the commutator is $(1^{\xi_1}, \dots, n^{\xi_n})$.
Define $F_{w_c, n}: S_n \times S_n \rightarrow \ZZ_{\geq 0}$ such that $F_{w_c, n} (\sigma, \tau) = \# \text{ fixed points of }[\sigma, \tau]$. Then, 
$$s_n(\sigma, \tau) =\sum_{\ell = 2}^n \xi_\ell = C_{w_c, n}(\sigma, \tau) - F_{w_c, n}(\sigma, \tau)$$
Using Theorem \ref{thm:moments}
$$ \EE (s_n) = \EE(C_{w_c,n}) - \EE(F_{w_c, n}) = \EE(C_{\eta, n}) - 1 - \frac{1}{n-1} + O\left(\frac{\log(n)}{n}\right)= \log(n) + \gamma - 1 +o(1) $$
where we use the fact that $\EE(F_{w_c, n}) = 1+1/(n-1)$ for $n > 4$ as shown by A. Nica in \cite{Nic}, Example 3.2.3. 
\end{proof}

\subsection{Distribution among strata} It is well known (for instance, from Theorem \ref{thm:shresthawang}) that for any stratum $\calH(\alpha_1, \dots, \alpha_s)$ of genus $g$ translation surfaces, $\STS_n \cap \calH(\alpha_1, \dots, \alpha_s) \neq \emptyset$ as long as $n \geq 2g-2+s$. A natural question then is to ask how random square-tiled surfaces distribute among the various strata. The next theorem answers this question, but before we state the theorem, we first recall the definition of the \textbf{Poisson-Dirichlet distribution} \cite{Gam}. Let $B_1, B_2, \dots $ be independent uniformly distributed random variables on $[0,1]$. Let $G = (G_1, G_2, \dots)$ be defined as,

$$G_1 = B_1; \quad G_2 = (1-B_1)B_2; \quad G_i = (1-B_1)(1-B_2) \dots (1-B_{i-1})B_i.$$

The random sequence $G$ describes the lengths of the pieces broken off in a random stick breaking process: start off with a unit length stick and then break off a piece of length $B_1$ on the left leaving a piece of length $(1-B_1)$. From this, break off a piece of length $B_2(1-B_1)$ on the left leaving $(1-B_1)(1-B_2)$, and so on. $G$ is then a a random variable on the infinite simplex
$$ \Delta = \{ x \in \RR^\infty: x_i \geq 0, \sum_i^\infty x_i = 1\}.$$
The distribution of the ordered version of $G$, $(G_{(1)}, G_{(2)}, \dots)$ with $G_{(1)} \geq G_{(2)} \geq \dots$ is what is known as the \textbf{Poisson-Dirichlet distribution}. We are now ready to state the next theorem:

\begin{theorem}[Distribution among strata]\label{thm:stratadist}The distribution of vertex angles for random square-tiled surfaces converges to Poisson-Dirichlet distribution.
\end{theorem}
\begin{proof}
In \cite{Gam}, Corollary 4.1, Gamburd shows that the distribution of normalized, ordered cycle lengths of a uniformly distributed permutation in $A_n$ converges to the Poisson-Dirichlet distribution as $n \rightarrow \infty$.

Applying Lemma \ref{lem:distconv} and the triangle inequality, we conclude that the distribution of normalized, ordered, cycle lengths of a $P_{w_c,n}$-distributed permutation in $A_n$ converges to the Poisson-Dirichlet distribution. Recall that the $P_{w_c, n}$-distribution in $A_n$ is precisely the distribution of random commutators obtained from two uniformly distributed random permutations from $S_n$. 

Then, finally, Proposition \ref{prop:TopFromCommutator} relates the cycle lengths of random commutators to the angles around the vertices of random STSs and we obtain the theorem.
\end{proof}

Given a fixed $n$, there are only finitely many strata $\calH(\alpha)$ in which STSs with $n$ squares can live. Ideally, we would like to be able to approximate the probability of falling in each of these strata. However, Theorem \ref{thm:stratadist} does not allow us to carry out such an approximation. Nevertheless, one can use Lemma \ref{lem:distconv} to calculate the asymptotic density as $n \rightarrow \infty$ of STSs in  principal strata (the strata corresponding to surfaces with all cone angles $4\pi$) or minimal strata (the strata corresponding to surfaces with exactly one cone point). 

\begin{proposition}\label{prop:principal}Let $S(\sigma, \tau) \in \STS_n$ given by $(\sigma, \tau) \in S_n \times S_n$. Then, as $n \rightarrow \infty$, 

$$\Pr[S(\sigma, \tau) \text{ belongs to a principal stratum}] \rightarrow 0;\qquad \Pr[S(\sigma, \tau) \text{ belongs to a minimal stratum}] \rightarrow 0.$$
\end{proposition}

\begin{proof}
Using Proposition \ref{prop:TopFromCommutator}, we begin by noting that $S(\sigma, \tau)$ is in a principal stratum (i.e., has all cone angles $4 \pi$) if and only if $S(\sigma, \tau)$ is connected and $[\sigma, \tau]$ is a product of disjoint transpositions (and 1-cycles). Similarly, $S(\sigma, \tau)$ is in a minimal stratum (i.e., has exactly one cone point) if and only if $S(\sigma, \tau)$ is connected and $[\sigma, \tau]$ is a non-trivial cyclic permutation.
Hence, using Lemma \ref{lem:distconv} and Proposition \ref{prop:transpositions},
\begin{align*}\Pr[S(\sigma, \tau) \text{ belongs to a principal stratum} ] &\leq \Pr[[\sigma, \tau] \text{ is a product of disjoint transpositions}]\\ &= \Pr[\eta \in A_n \text{ is a product of disjoint transpositions}] + O(n^{-1})\\&\rightarrow 0.
\end{align*}
Similarly,
\begin{align*}\Pr[S(\sigma, \tau) \text{ belongs to a minimal stratum} ] &\leq \Pr[[\sigma, \tau] \text{ is a non-trivial cyclic permutation}]\\ &= \Pr[\eta \in A_n \text{ is a non-trivial cyclic permutation}] + O(n^{-1})\\&\rightarrow 0.
\end{align*}
\end{proof}

It is interesting to compare Proposition \ref{prop:principal} with what happens in the general case of translation surfaces of a fixed genus. It is known that a stratum $\calH(\alpha_1, \dots, \alpha_s)$ of genus $g$ translation surfaces with $s$ cone points is a complex orbifold of dimension $2g-2+s$. Hence the principal stratum (the one with the most cone points) has the highest dimension among genus $g$ strata, and the minimal stratum has the lowest. So, if we fix the genus, we expect a generic translation surface to fall in the principal stratum. However, when the genus is not fixed, and we consider simply square-tiled surfaces, Proposition \ref{prop:principal} asserts the opposite. In fact, the computation in Proposition \ref{prop:transpositions} suggests that the  probability of a random STS falling in a minimal stratum is higher than the probability of it falling in a principal stratum.

\section{Proof of Geometry result}\label{sec:geometry}

This section will be devoted to proving the following theorem:
\holonomy

We begin with the following lemma which gives sufficient criterion on the number of squares for an STS to be a visibility torus. For connected surfaces, this lemma is part (3) of Theorem \ref{thm:shresthawang}.

\begin{lemma}\label{lem:viscriterion}Let $S$ be a (not necessarily connected) $n$-square-tiled surface with $s$ cone points. Let $g = 1-\chi(S)/2$ where $\chi(S)$ is the Euler characteristic. If $4g + 2s - 4 > n$, then $S$ is a visibility torus.
\end{lemma} 
\begin{proof}
Assume $S(\sigma, \tau)$ has $\ell \geq 1$ connected components with genera $g_1, \dots, g_\ell$ and Euler characteristic $\chi_1, \dots, \chi_\ell$. Also let $s_1, \dots, s_\ell$ be the cone points and $n_1, \dots, n_\ell$ be the number of squares in each of the components. Then,
$$g = 1 - \frac{\chi(S)}{2} = 1 - \frac{\sum_{i=1}^\ell \chi_i}{2} = \sum_{i=1}^\ell g_i - (\ell-1)$$
since $g_i = 1 - \frac{\chi_i}{2}$. Now, assume $4g_i + 2s_i - 4 \leq n_i$ for each $i$. Note $\sum_i n_i = n$.  This then implies, 
$$4\sum_{i=1}^\ell  g_i + 2\sum_{i=1}^\ell s_i - 4l = 4g+ 2 s - 4 \leq n,$$
a contradiction. So there must exist a connected component for which $4g_i + 2s_i - 4 > n$. By Theorem \ref{thm:shresthawang} part (3), this connected component is a visibility torus. Hence, regardless of what happens in other components, $S(\sigma, \tau)$ is a visibility torus.
\end{proof}

Next we give the proof of the Holonomy Theorem.
\begin{proof}[Proof of \ref{thm:holonomy}](\emph{Holonomy}) We first prove that
$\Pr[S(\sigma, \tau) \text{ is a holonomy torus}] = \frac{1}{e} + O(n^{-1})$. Define,
$$X_n := \{(\sigma, \tau) \in S_n \times S_n| [\sigma, \tau] \text{ is a derangement}\},$$
$$ Y_n := \{(\sigma, \tau) \in S_n \times S_n| \text{ fixed points of }[\sigma, \tau] = \text{ fixed points of }\sigma = \text{ fixed points of }\tau \}.$$
By Proposition \ref{prop:HolFromComb}, 
\begin{align*}
\Pr[S(\sigma, \tau) \text{ is a holonomy torus}] = \Pr[(\sigma, \tau) \in X_n] +  \Pr[(\sigma, \tau) \in Y_n] 
\end{align*}
We note that for all $(\sigma, \tau) \in Y_n$, $\ideal{\sigma, \tau}$ is not transitive (for instance, the common fixed points are fixed by the entire group). So, by Dixon \cite{Dix}, we know, $\Pr[(\sigma, \tau) \in Y_n] = O(n^{-1})$. 
On the other hand, by Lemma \ref{lem:distconv} and Proposition \ref{prop:derangementAn},  $\Pr[(\sigma, \tau) \in X_n] = e^{-1} + O(n^{-1})$. So, we conclude,
$$\Pr[S(\sigma, \tau) \text{ is a holonomy torus}] = \frac{1}{e} + O(n^{-1}).$$

(\emph{Visibility}) Next we prove that,
$\Pr[S(\sigma, \tau) \text{ is a visibility torus}] = 1 - O\left(\frac{n}{2^{n/2}}\right).$

Recall the random variables $G_n:=G_{n,4}$ and $s_n:=s_{n,4}$ that count the genus (in the connected case) and the number of cone points respectively of a random $S \in \STS_n$. By Lemma \ref{lem:viscriterion} we observe that 
$$\Pr[S(\sigma, \tau) \text{ is a visibility torus}]> \Pr[4 G_n + 2s_n - 4 > n]$$

Now, $\Pr[4 G_n + 2s_n - 4 > n] > \Pr[4 G_n - 4 > n] = \Pr[G_n > (n+4)/4]$. However,
\begin{align*}
 \Pr\left[G_n > \frac{n+4}{4}\right] &= \Pr\left[\frac{n}{2} - \frac{C_{w_1,n}}{2} + 1 > \frac{n+4}{4}\right] = 1-\Pr\left[C_{w_1,n}\geq \frac{n}{2}\right] =  1 - O\left(\frac{n}{2^{n/2}}\right)
\end{align*}
where the last equality is due to Lemma \ref{lem:largedev}. 
So,
 $$\Pr[S(\sigma, \tau) \text{ is a visibility torus}] \geq \Pr\left[G_n > \frac{n+4}{4}\right] = 1 - O\left(\frac{n}{2^{n/2}}\right)$$
which implies $\Pr[S(\sigma, \tau) \text{ is a visibility torus}] = 1 - O\left(\frac{n}{2^{n/2}}\right)$. Therefore, we conclude that a random STS is asymptotically almost surely a visibility torus.
\end{proof}

\section{Generalization to Polygon-tiled Surfaces}\label{sec:PTS}
In this section we generalize the theory of square-tiled surfaces to polygon-tiled surfaces. In particular, we generalize the Genus Theorem and Theorem \ref{thm:moments}. 

\subsection{Randomizing model for PTSs}\label{sec:randmodelGEN}
\begin{wrapfigure}{r}{0.5\textwidth}
\begin{tikzpicture}
    [oct/.style=
  {shape=regular polygon, regular polygon sides=10, draw, minimum width=4cm, shape border rotate=90}]
  \node[oct] at (0,0) (i) {$i$};
%
\node at (12*18:3cm){$\sigma_5(i)$};
\node at (8*18:3cm){$\sigma_2(i)$};
\node at (16*180/10:3cm){$\sigma_3(i)$};
\node at (0:3cm){$\sigma_1(i)$};
\node at (4*180/10:3cm){$\sigma_4(i)$};

\foreach \vert [count = \i] in {$v_1$, $v_2$, $v_3$, $v_4$, $v_5$, $v_6$, $v_7$, $v_8$, $v_9$, $v_{10}$}{
\node at (\i*2*18+18-36:1.7cm){\vert};
}
\foreach \edge [count = \i] in {$e_1$,$e_2$,$e_3$,$e_4$,$e_5$,$e_6$,$e_7$,$e_8$,$e_9$,$e_{10}$}
{ \node at (\i*2*18 - 36:1.7cm){\edge};
}
%
%

\begin{scope}
\clip (-3.5,-3.5) rectangle (3,3);
\node[oct] at (0:3.8) {};
\node[oct] at (8*18:3.8){};
\node[oct] at (16*18:3.8){};
\node[oct] at (4*18:3.8){};
\node[oct] at (12*18:3.8){};
\end{scope}


	\end{tikzpicture}
	\caption{Associating permutations describing the gluings of a decagon-tiled surface. Since $e_6$ is opposite to $e_1$, the decagon glued to side $e_6$ is $\sigma_1^{-1}(i)$ and so on for the other sides whose neighbors are not shown.}
	\label{fig:2kscheme}
\vspace{-0.2cm}
\end{wrapfigure}
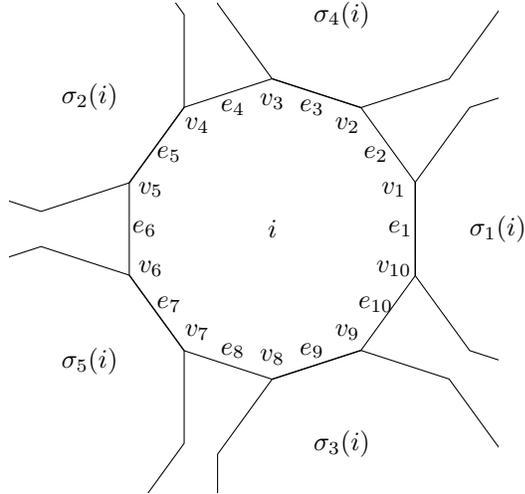
Similar to square-tiled surfaces, for any $S \in \PTS_{n,2k}$ we can encode the gluings by a $k$-tuple of permutations. First label the polygons $\{1, \dots, n\}$.  Next, for each polygon, label and orient the edges $e_1, \dots, e_{2k}$, counterclockwise starting at the right vertical edge. Also, we label the corners $v_1, \dots, v_{2k}$ counterclockwise starting at the top of edge $e_1$. Then, define $\sigma_1 \in S_n$ such that $\sigma_1(i)$ is the polygon glued to edge $e_1$ of polygon $i$. This ensures that the polygon $\sigma_1(i)$ is always the neighbor to the right of polygon $i$, directly generalizing the square-tiled case. Next, $\sigma_2 \in S_n$ is defined such that $\sigma_2(i)$ is the polygon glued to edge $e_k$ of polygon $i$ and  $\sigma_3 \in S_n$ such that $\sigma_3(i)$ is the polygon glued to edge $e_{2k-1}$. Continue on so that $\sigma_{l}(i)$ is the polygon glued to edge $e_{(l-1)(k-1)+1 \mod 2k}$ of polygon $i$. This is illustrated in Figure \ref{fig:2kscheme} for $k=5$. Note that, just as we needed only two permutations in the square-tiled case, it suffices to describe the gluings associated to $k$ of the non-parallel sides. The gluings on the opposite parallel sides to the ones just described are encoded by the corresponding inverse permutations. See Figure \ref{fig:hexocttiled} for some examples.

Conversely, given a $k$ tuple of permutations, $(\sigma_1, \dots, \sigma_k) \in \prod^k S_n$, one can associate to such a tuple a polygon-tiled surface $S(\sigma_1, \dots, \sigma_k) \in \PTS_{n,2k}$. Again, this surface is connected if and only if the subgroup generated by the permutations $\ideal{\sigma_1, \dots, \sigma_k} \subset S_n$, is a transitive subgroup and hence, a uniformly distributed $k$-tuple of permutations in $S_n$ gives a connected surface asymptotically almost surely as $n \rightarrow \infty$. So, $\prod^{k} S_n$ with the uniform measure  is a combinatorial model for random polygon-tiled surfaces built out of $n$ $2k$-gons.

\subsection{Genus of PTSs from the randomizing model}

Similar to the square-tiled case, in order to determine the genus of a random PTS in $\PTS_{n, 2k}$ we start with the Euler characteristic formula. For any $S \in \PTS_{n, 2k}$, we have $n$ faces and $kn$ edges so that the genus is determined by the number of the vertices. 

To count the number of vertices, we can use the next proposition. Essentially, it tells us that the products $c:=\sigma_1 \dots \sigma_k  \sigma_1^{-1} \dots \sigma_k^{-1}$, $a:=\sigma_1 \dots \sigma_k$ and $b:=\sigma_1^{-1} \dots \sigma_k^{-1}$ capture the vertex equivalence classes when $(\sigma_1, \dots, \sigma_k)$ describes a polygon-tiled surface. When $k=2$, note that $c$ is exactly the commutator and the following proposition is seen as a direct generalization of Proposition \ref{prop:vertexgluing}.

\begin{proposition}[Combinatorial condition for identifying vertices]\label{prop:vertexgluingGEN}Let $S \in \PTS_{n,2k}$ be given by $(\sigma_1, \dots, \sigma_k) \in \prod^k S_n$. Let $a, b, c$ be the products of permutations as defined above.
\begin{enumerate}
\item If $k$ is even, corner $v_{k+1}$ of polygon $i$ is glued to corner $v_{k+1}$ of polygon $j$ if and only if there exists $\ell\geq 1$ such that $c^\ell(i) = j$. 
\item If $k$ is odd, corner $v_{k+1}$ of polygon $i$ is glued to corner $v_{k+1}$ of polygon $j$ if and only if there exists $\ell\geq 1$ such that $a^\ell(i) = j$. Similarly, corner $v_1$ of polygon $i$ is glued to corner $v_1$ of polygon $j$ if and only if there exists $\ell\geq 1$ such that $b^\ell(i) = j$.
\end{enumerate} 
\end{proposition}
\begin{proof}
When $k$ is even, the proof of the general case is analogous to the proof of Proposition \ref{prop:vertexgluing}. Geometrically, the action of $c$ on polygon $i$ corresponds to winding around the vertex represented by corner $v_{k+1}$ of $i$ as in Figure \ref{fig:vertexfromcommutator}. 

%
%
%
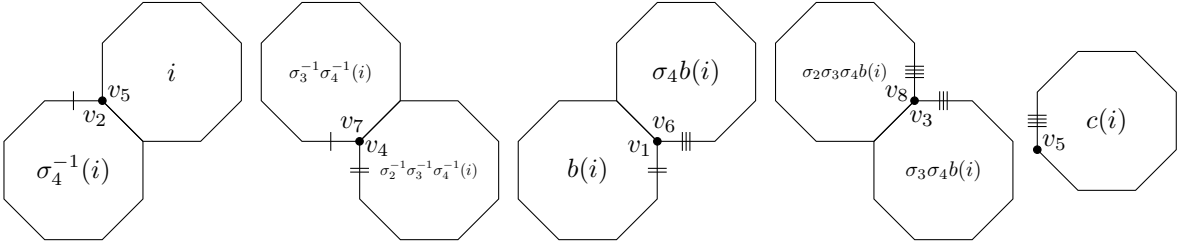
\begin{figure}[h!!]
\begin{minipage}{0.2\textwidth}
\centering
\begin{tikzpicture}[oct/.style=
 {shape=regular polygon, regular polygon sides=8, draw, minimum width=2cm}]
 \node[oct] at (0,0) (i) {$i$};
 \node[oct] at (-6*22.5:1.85cm) {$\sigma_4^{-1}(i)$};
 
 \node at (-7*22.5:0.75cm){$v_5$};
  \draw [fill] (-7*22.5:1cm) circle 	[radius=0.05];
  
  \begin{scope}[shift={(-6*22.5:1.85cm)}]
  \node at (3*22.5:0.75cm){$v_2$};
  \draw (4*22.5:0.8cm) -- (4*22.5:1.05cm);
  
  \end{scope}

\end{tikzpicture}
\end{minipage}
\begin{minipage}{0.2\textwidth}
\centering
\begin{tikzpicture}[oct/.style=
 {shape=regular polygon, regular polygon sides=8, draw, minimum width=2cm}]
 \node[oct] at (0,0) (i) {};
 \node[scale=0.7] at (0,0) {$\sigma_3^{-1}\sigma_4^{-1}(i)$};
 \node[oct] at (-2*22.5:1.85cm) {};
 \node[scale=0.6] at (-2*22.5:1.85cm){$\sigma_2^{-1}\sigma_3^{-1}\sigma_4^{-1}(i)$};
 
 \node at (-3*22.5:0.75cm){$v_7$};
 \draw [fill] (-3*22.5:1cm) circle 	[radius=0.05];
 \draw (-4*22.5:0.8cm) -- (-4*22.5:1.05cm);
 \begin{scope}[shift={(-2*22.5:1.85cm)}]
  \node at (7*22.5:0.75cm){$v_4$};
  \draw [shift={(0,0.025)}] (8*22.5:0.8cm) -- (8*22.5:1.05cm);
  \draw [shift={(0,-0.025)}] (8*22.5:0.8cm) -- (8*22.5:1.05cm);

  \end{scope}

\end{tikzpicture}
\end{minipage}
\begin{minipage}{0.2\textwidth}
\centering
\begin{tikzpicture}[oct/.style=
 {shape=regular polygon, regular polygon sides=8, draw, minimum width=2cm}]
 \node[oct] at (0,0) (i) {$\sigma_4 b(i)$};
 \node[oct] at (-6*22.5:1.85cm) {$b(i)$};
 
 \node at (-5*22.5:0.75cm){$v_6$};
 \draw [fill] (-5*22.5:1cm) circle 	[radius=0.05];
 \draw [shift={(-0.05,0)}](-4*22.5:0.8cm) -- (-4*22.5:1.05cm);
 \draw (-4*22.5:0.8cm) -- (-4*22.5:1.05cm);
 \draw [shift={(0.05,0)}](-4*22.5:0.8cm) -- (-4*22.5:1.05cm);
 \begin{scope}[shift={(-6*22.5:1.85cm)}]
  \node at (1*22.5:0.75cm){$v_1$};
  \draw [shift={(0,0.025)}] (0*22.5:0.8cm) -- (0*22.5:1.05cm);
  \draw [shift={(0,-0.025)}] (0*22.5:0.8cm) -- (0*22.5:1.05cm);

  \end{scope}

\end{tikzpicture}
\end{minipage}
\begin{minipage}{0.2\textwidth}
\centering
\begin{tikzpicture}[oct/.style=
 {shape=regular polygon, regular polygon sides=8, draw, minimum width=2cm}]
 \node[oct] at (0,0) (i) {};
 \node[scale=0.7] at (0,0) {$\sigma_2\sigma_3\sigma_4b(i)$};
 
 \node[oct] at (-2*22.5:1.85cm) {};
 \node[scale=0.8] at (-2*22.5:1.85cm){$\sigma_3\sigma_4b(i)$};

 \node at (-1*22.5:0.75cm){$v_8$};
 \draw [fill] (-1*22.5:1cm) circle 	[radius=0.05];
 \draw [shift={(0,0.075)}](0:0.8cm) -- (0:1.05cm);
 
 \draw [shift={(0,0.025)}](0:0.8cm) -- (0:1.05cm);
 \draw [shift={(0,-0.025)}](0:0.8cm) -- (0:1.05cm);
 \draw [shift={(0,-0.075)}](0:0.8cm) -- (0:1.05cm);

 \begin{scope}[shift={(-2*22.5:1.85cm)}]
  \node at (5*22.5:0.75cm){$v_3$};
 \draw [shift={(-0.05,0)}](4*22.5:0.8cm) -- (4*22.5:1.05cm);
 \draw (4*22.5:0.8cm) -- (4*22.5:1.05cm);
 \draw [shift={(0.05,0)}](4*22.5:0.8cm) -- (4*22.5:1.05cm);
  
 \end{scope}
 
\end{tikzpicture}
\end{minipage}
\begin{minipage}{0.1\textwidth}
\centering
\begin{tikzpicture}[oct/.style=
 {shape=regular polygon, regular polygon sides=8, draw, minimum width=2cm}]
 \node[oct] at (0,0) (i) {$c(i)$};
 \node at (9*22.5:0.75cm){$v_5$};
  \draw [fill] (9*22.5:1cm) circle 	[radius=0.05];
 \draw [shift={(0,0.075)}](8*22.5:0.8cm) -- (8*22.5:1.05cm);
 
 \draw [shift={(0,0.025)}](8*22.5:0.8cm) -- (8*22.5:1.05cm);
 \draw [shift={(0,-0.025)}](8*22.5:0.8cm) -- (8*22.5:1.05cm);
 \draw [shift={(0,-0.075)}](8*22.5:0.8cm) -- (8*22.5:1.05cm);

\end{tikzpicture}
\end{minipage}

\caption{Following corner $v_5$ of $i$ under the generalized commutator $c$ in an octagon-tiled surface. Only the relevant gluings around the corners are shown through which we see that $v_5$ of $i$ is glued to $v_2$ of $\sigma_4^{-1}$ is glued to $v_7$ of $\sigma_3^{-1}\sigma_4^{-1}$ and so on until $v_5$ of $j = c(i)$.}
\label{fig:vertexfromcommutator}

\end{figure}

The case for odd $k$ is similar. Using a geometric viewpoint, the action of $a$ on $i$ corresponds to winding around the vertex represented by corner $v_{k+1}$ and the action of $b$ on $i$ corresponds to winding around the vertex represented by corner $v_1$.\end{proof}



Proposition \ref{prop:vertexgluingGEN} naturally leads to the following observation which is a direct generalization of Proposition \ref{prop:TopFromCommutator}.

\begin{proposition}[Cycle type determines cone point data]\label{prop:TopFromCommutatorGEN}Let $S \in \PTS_{n,2k}$ be given by $(\sigma_1, \dots, \sigma_k) \in \prod^k S_n$. 
\begin{enumerate}
\item If $k$ is even and the cycle type of $\sigma_1 \dots \sigma_k \sigma_1^{-1} \dots \sigma_k^{-1}$ is $(1^{\xi_1}, 2^{\xi_2}, \dots, n^{\xi_n})$, then $S$ has $\xi_l$ vertices with angle $2\pi(k-1) \ell$ for each $\ell=1, \dots, n$.

\item If $k$ is odd and the cycle type of $\sigma_1 \dots \sigma_k$ is $(1^{\xi_1}, 2^{\xi_2}, \dots, n^{\xi_n})$ and the cycle type of $\sigma_1^{-1} \dots \sigma_k^{-1}$ is $(1^{\beta_1}, 2^{\beta_2}, \dots, n^{\beta_n})$, then $S$ has $\lambda_l + \beta_l$ vertices with angle $\pi(k-1)\ell$ for each $\ell = 1, \dots, n$
\end{enumerate}
\end{proposition} 
\begin{proof}

When $k$ is even, every vertex in $S$ is represented by corner $v_{k+1}$ of some polygon $i$. This is because every vertex is a pre-image of the lone vertex in the base $2k$-gon surface which can be represented by corner $v_{k+1}$. By Proposition \ref{prop:vertexgluing} disjoint cycles of $c$ are in bijection with the distinct vertices. Following the gluings from $i$ to $c(i)$ corresponds to winding  around the vertex and traversing $2\pi(k-1)$ angle each. Hence, if the associated cycle's length is $\ell$, then the vertex angle is $2 \pi(k-1)\ell$.

The case for odd $k$ is similar. In this case the base surface has two branch points, say $p$ and $q$. One is represented by corner $v_{k+1}$ (say, $p$ is) and the other by $v_1$. Hence, each vertex in the surface that is in the pre-image (under the branch covering map) of $p$ is represented by corners $v_{k+1}$ and vertices in the pre-image of $q$ are represented by corners $v_1$. By Proposition \ref{prop:vertexgluing} cycles of $a$ are then in bijection with vertices in the pre-image of $p$ while cycles of $b$ are in bijection with pre-images of $q$. Following the gluings from $i$ to $a(i)$ and $i$ to $b(i)$ each corresponds to winding around the respective vertices and traversing $\pi(k-1)$ angle each. \end{proof}

Propositions \ref{prop:vertexgluingGEN} and \ref{prop:TopFromCommutatorGEN} motivate the study of the permutations $a = \sigma_1 \dots \sigma_k, b = \sigma_1^{-1}, \sigma_2^{-1}, \dots, \sigma_k^{-1}$ and $c = ab$ in order to understand the genus of random PTSs. To study these products we briefly introduce the general theory of word maps in the next section.

\subsection{Word maps}\label{sec:wordmaps}

Recall that for the square-tiled case, we considered the probability distribution on $A_n$ induced by the commutator map $w_c: S_n\times S_n \rightarrow S_n$. More generally, let $w$ be a word in $F_k = F(x_1, \dots, x_k)$ the free group on $k$ generators. Given a group $G$, we can consider a \textbf{word map}, $w:\prod^k G \rightarrow G$ associated to the word $w \in F_k$ given by substitution: for $w = \prod_{j=1}^r x_{i_j}^{\epsilon_j}$ with $\epsilon_j =\pm 1$, 
$$ w(g_1, \dots, g_k) = \prod_{j=1}^r g_{i_j}$$
For example, if $w \in F(x_1,x_2)$ is the word $w = [x_1,x_2] = x_1x_2x_1^{-1}x_2^{-1}$, then $w(g_1, g_2) = [g_1, g_2]$.
Although word maps are generally defined, we will assume $G$ is finite for our purposes. For any $g \in G$, let $N_w(g) = \#w^{-1}(g)$. Then, any word map induces a probability measure on $G$, which we will denote $P_w$ given by 
$$ P_w(g) = \frac{N_w(g)}{\#G^k}$$
The function $N_w$ satisfies some basic properties:
\begin{itemize}
\item $N_w$ is a class function
\item If $w$ and $w'$ are equivalent under $\Aut(F_k)$, then $N_w \equiv N_{w'}$
\item If $w = x_1$, then $N_w \equiv \#G^{k-1}$
\end{itemize}

Since class functions on a finite group can be written as linear combinations of irreducible characters, we can write 
$ N_w = \sum_{\chi \in \Irr(G)}N_w^\chi\cdot \chi$ for $N_w^\chi \in \CC$. The coefficients $N_w^\chi$ are called the \textbf{Fourier coefficients} of the word map $w$. The study of these Fourier coefficients goes back to Frobenius \cite{Frob} who proved that 
$ N_{[x_1,x_2]}^\chi = \frac{\# G}{\chi(1)}$ and hence gave the formula in (\ref{eqn:probspecialform}). More generally, we are concerned with the Fourier coefficients of the word maps associated to words 
$$w_{c} := x_1 \dots x_k x_1^{-1}\dots x_k^{-1} \hspace{1cm} w_{a} := x_1x_2 \dots x_k \hspace{1cm} w_{b} := x_1^{-1}x_2^{-1} \dots x_k^{-1}.$$
The next proposition simplifies the task to simply understanding $w_c$.
\begin{proposition}\label{prop:w1w2equiv}$N_{w_a} \equiv N_{w_b} \equiv \#G^{k-1}$.
\end{proposition}
\begin{proof}
Replacing $x_i$ with its inverse is an elementary Nielsen transformation and an automorphism of $F_k$. So, $w_b  \sim w_a$ under $\Aut(F_k)$. Similarly, replacing $x_1$ by $x_1x_2$ is an elementary Nielsen transformation. Inductively, $x_1 \sim x_1x_2 \dots x_k = w_a \sim w_b$. Hence, $N_{w_a}=N_{w_b} = \#G^{k-1}$.\end{proof}
As a consequence of Proposition \ref{prop:w1w2equiv}, the induced probability measures $P_{w_a, n, 2k}$ and $P_{w_b, n, 2k}$ on $S_n$ induced by $w_j:\prod^k S_n \rightarrow S_n$ (for $j=a,b$) are exactly the uniform measure. The case of $P_{w_c, n, 2k}$, induced by $w_c: \prod^k S_n \rightarrow A_n$ is slightly more complicated.
The coefficients of $N_{w_c}$ are given by the following theorem of Tambour (which was generalized by Parzanchevski-Schul \cite{ParSchul}).

\begin{theorem}[Tambour \cite{Tam}]\label{thm:fouriercoeffs} For $w_c = x_1x_2 \dots x_kx_1^{-1}\dots x_k^{-1} \in F(x_1, \dots, x_k)$, the function $N_{w_c}$ can be expanded into irreducible characters as,
$$N_{w_c} \equiv \begin{cases}\sum_{\chi \in \Irr(G)} \frac{\#G^{k-1}}{\chi(1)^{k-1}} \cdot \chi \text{ for even }k \geq 2 \\ \hspace{1cm} \\ \sum_{\chi \in \Irr(G)} \frac{\#G^{k-1}}{\chi(1)^{k-2}} \cdot \chi \text{ for odd }k \geq 3 \end{cases}$$
\end{theorem}  

Note that $P_{w_c, n, 4} = P_{w_c}$ and Theorem \ref{thm:fouriercoeffs} generalizes Frobenius's formula for the Fourier coefficients of the commutator word map.  

Now, given a random polygon-tiled surface $(\sigma_1, \dots, \sigma_k) \in \prod^k S_n$, using Proposition \ref{prop:vertexgluingGEN} we see that the number of vertices equals the number of cycles of the permutation  $w_c(\sigma_1, \dots \sigma_k)$ if $k$ is even, and the sum of the number of cycles of $w_a(\sigma_1\dots \sigma_k)$ and $w_b(\sigma_1 \dots \sigma_k)$ if $k$ is odd. Hence, we define the random variables, for every $k \geq 2$ and $n \geq 1$, 
\begin{equation}\label{eq:numcycles}C_{w_i, n, 2k} = \# \text{ cycles of }w_i(\sigma_1, \dots, \sigma_k)
\end{equation}
where $(\sigma_1, \dots, \sigma_k) \in \prod^kS_n$ is uniformly distributed. Then define $G_{n,2k}: \prod^k S_n \rightarrow \NN$ to be the random variable
$$ G_{n,2k} = \begin{cases}\frac{(k-1)n}{2} - \frac{C_{w_c, n, 2k}}{2} + 1 \text{ for even }k\geq 2\\
\hspace{1cm}\\
\frac{(k-1)n}{2} - \frac{C_{w_a, n, 2k}+ C_{w_b, n, 2k}}{2} + 1 \text{ for odd }k\geq 3
\end{cases} $$
We will drop the subscript $2k$ from the random variables when it is implied by context.

%

\subsection{Generalization of preliminary lemmas} We now generalize the appropriate lemmas from the square-tiled case. In the first lemma we calculate the generating function for the random variable $C_{w_c,n,2k}$ which counts the number of cycles of a $P_{w_c, n, 2k}$-distributed random permutation from $A_n$. Note the case $k=2$ is exactly Lemma \ref{lem:genfcn}. 

\begin{lemma}\label{lem:genfcnGEN}The generating function for $C_{w_c,n,2k}$ is given by,

$$ G(q):= \sum_{\sigma_1, \dots, \sigma_k \in S_n} q^{C_{w_c,n,2k}(\sigma_1, \dots, \sigma_k)} = \begin{cases}(n!)\sum_{\lambda \vdash n} H_\lambda^{k-2}\prod_{r \in \lambda}(q+\cont(r)) \text{ for }k\geq 2 \text{ even}\\ \hspace{1cm} \\(n!)^2\sum_{\lambda \vdash n} H_\lambda^{k-3}\prod_{r \in \lambda}(q+\cont(r)) \text{ for }k\geq 3 \text{ odd} \end{cases}$$

where 
\begin{itemize}
\item the sum is over unordered partitions $\lambda$ of $n$,
\item $H_\lambda$ is the product of hook lengths of $\lambda$, 
\item $\cont(r)$ is the content of square $r$.
\end{itemize}
\end{lemma}

\begin{proof}In this proof we will use some of the combinatorial tools introduced in Appendix \ref{sec:combinatorics}.

We begin by applying the Frobenius characteristic map (equation (\ref{eq:frobeniuscharmap1})) to the class function $N_{w_c}$ which records the number of pre-images of an element under the word map $w_c$. 
$$\ch N_{w_c} = \sum_{\mu \vdash n} z^{-1}_\mu N_{w_c}(\mu) p_\mu = \frac{1}{n!}\sum_{\sigma \in S_n} N_{w_c}(\sigma)p_{\cyc(\sigma)} = \frac{1}{n!}\sum_{\sigma_1, \dots, \sigma_k \in S_n}p_{\cyc(w_c(\sigma_1, \dots, \sigma_k))}  $$
where $p_\mu$ is the power sum symmetric function associated to the partition $\mu$ and $\cyc(\sigma)$ is the partition associated to the cycle type of $\sigma$. 

However, in terms of the Schur symmetric functions (equation (\ref{eq:frobeniuscharmap2})), and after applying Theorem \ref{thm:fouriercoeffs}, we then have
$$\ch N_{w_c} = \begin{cases} \sum_{\lambda \vdash n} \frac{|S_n|^{k-1}}{\chi^\lambda(1)^{k-1}} s_\lambda \text{ for even }k \geq 2 \\
\hspace{1cm} \\
\sum_{\lambda \vdash n} \frac{|S_n|^{k-1}}{\chi^\lambda(1)^{k-2}} s_\lambda \text{ for odd }k \geq 3\end{cases}$$

Equating the two expressions for $\ch N_{w_c}$ and using Lemma \ref{lem:hooklength} so that $H_\lambda^{k-1} = \frac{|S_n|^{k-1}}{\chi^\lambda(1)^{k-1}}$, we get

\begin{equation}\label{eq:almosgen}\frac{1}{n!}\sum_{\sigma_1, \dots, \sigma_k \in S_n}p_{\cyc(w_c(\sigma_1, \dots, \sigma_k))}= \begin{cases}
\sum_{\lambda \vdash n } H_\lambda^{k-1} s_\lambda \text{ for even }k \geq 2 \\
\hspace{1cm}\\
\sum_{\lambda \vdash n } H_\lambda^{k-1} \chi^\lambda(1) s_\lambda  \text{ for odd }k \geq 3\end{cases}
\end{equation}

Now, taking $x_1 = \dots = x_q = 1$ and $0$ for the rest of the indeterminates, on the left we get
$$ \frac{1}{n!} \sum_{\sigma_1, \dots, \sigma_k \in S_n} q^{C_{w_c,n,2k}(\sigma_1, \dots, \sigma_k)}$$
Using Lemma \ref{lem:schurspec} we get  $s_\lambda(1, 1, 1, \dots, 1) = \prod_{r \in \lambda} \frac{q+ \cont(r)}{h(r)}$ where $\cont(r)$ is the content and $h(r)$ is the hook length of square $r$ which simplifies the right hand side and equation (\ref{eq:almosgen}) then becomes,
$$ \sum_{\sigma_1, \dots, \sigma_k \in S_n} q^{C_{w_c,n,2k}(\sigma_1, \dots, \sigma_k)} = \begin{cases} n!\sum_{\lambda \vdash n} H_\lambda^{k-2} \prod_{r \in \lambda}(q+\cont(r)) \text{ for even }k \geq 2 \\
\hspace{1cm}\\ (n!)^2\sum_{\lambda \vdash n } H_\lambda^{k-3} \prod_{r \in \lambda}(q+\cont(r))  \text{ for odd }k \geq 3\end{cases}$$
which is the generating function required.\end{proof}

We next state the generalization of Lemma \ref{lem:largedev} to the case of the random variables $C_{w_c, n, 2k}$.

\begin{lemma}[Large deviations] \label{lem:largedevGEN}Let $C_{w_c,n,2k}$ be as before. Then,
$\Pr(C_{w_c,n,2k} \geq t) = O\left(\frac{n}{2^t}\right)$
\end{lemma}

The proof follows the proof of Lemma \ref{lem:largedev}. We start with the generating function obtained from Lemma \ref{lem:genfcnGEN} which is first converted into a probability generating function and evaluated at $q = 2$ to give an upper bound for $\Pr(C_{w_c,n, 2k} \geq t)$. Next, note that for any partition $\lambda \vdash n$, we have $\frac{H_\lambda}{n!} = \frac{n!}{\chi^\lambda(1) n!} \leq 1$. This then implies that $\Pr(C_{w_c,n,2k} \geq t) \leq \frac{1}{2^tn!} \sum_{j=0}^{\floor{n/2}} \prod_{r \in (n-j, j)}(2 + \cont(r))$ which is exactly the intermediate bound in the proof of Lemma \ref{lem:largedev}.

Similarly, we can generalize Lemma \ref{lem:distconv} and obtain the following lemma for the general case of polygon-tiled surfaces.
\begin{lemma}\label{lem:distconvGEN} Let $U_{A_n}$ and $P_{w_c, n, 2k}$ be as previously defined. Then,
$$||P_{w_c, n, 2k} - U_{A_n}|| = \begin{cases} O(n^{-k+1}) \text{ when }k\geq 2\text{ even}\\ \hspace{1cm}\\  O(n^{-k+2})  \text{ when }k\geq 3\text{ odd} \end{cases}$$
\end{lemma}

To obtain this lemma from the proof of Lemma \ref{lem:distconv}, we use the general formula obtained by Tambour in Theorem \ref{thm:fouriercoeffs} instead of Frobenius's formula for the number of pre-images under the commutator map. In particular, under Tambour's formula,  equation (\ref{eqn:probspecialform}) generalizes to 
\begin{equation}P_{w_c, n, 2k}(\pi) =  \frac{1}{n!}\sum_{\chi \in \Irr(S_n)} \frac{\chi(\pi)}{(\chi(1))^\epsilon}
\end{equation}
where $\epsilon = k-1$ when $k$ is even and $\epsilon = k-2$ when $k$ is odd. The rest of the proof follows closely the one in Lemma \ref{lem:distconv} with this modification.

\subsection{Moments of vertex count in PTSs}
We now state the generalization of Theorem \ref{thm:moments} to PTSs. 
\begin{theorem}[Moments of vertex count]\label{thm:momentsGEN}

Let $p(x)$ be a polynomial of degree $\ell$ that is non-negative and non-decreasing for $x \geq 0$. Then,
$$\EE[p(C_{w_c,n,2k})] = \EE[p(C_{\eta, n})] + \begin{cases} O\left(\frac{\log^\ell(n)}{n^{k-1}}\right)\text{ for even }k\geq 2 \\\hspace{1cm}\\ O\left(\frac{\log^\ell(n)}{n^{k-2}}\right) \text{ for odd }k\geq 3 \end{cases}$$ 
$$ \EE[p(C_{w_i, n, 2k})] = \EE[p(C_{\pi, n})] \text{ for }i=a,b$$

\end{theorem} 
The proof is analogous to the proof of the square-tiled case, but uses Lemmas \ref{lem:distconvGEN} and \ref{lem:largedevGEN} instead of Lemmas \ref{lem:distconv} and \ref{lem:largedev}.

The statistics of the genus random variable $G_{n, 2k}$ directly follows from Theorem \ref{thm:momentsGEN}. 
\begin{corollary}\label{cor:expectedGEN}
The expected value and variance of the genus random variable $G_{n, 2k}$ are given as follows:
$$\begin{array}{c|c|c}
 & \text{Even } k\geq 2 & \text{Odd }k \geq 3\\[8pt]
\EE[G_{n,2k}] & \frac{(k-1)n}{2} - \frac{\log(n)}{2} - \frac{\gamma}{2} + 1 + o(1) & \frac{(k-1)n}{2} - \log(n) - \gamma  + 1 + o(1) \\[8pt]
\Var[G_{n,2k}] &\frac{\log(n)}{4}+ \frac{\gamma}{4} - \frac{\pi^2}{24} + o(1) & \frac{\log(n)}{2} + \frac{\gamma}{2} - \frac{\pi^2}{12} + o(1)\\[10pt]
\end{array}$$
\end{corollary}
Similarly, the local central limit theorem from $G_{n, 2k}$ for even $k$ also follows from previous observations of Chmutov-Pittel as explained for the square-tiled case. When $k$ is odd, $G_{n, 2k}$ is given by a linear combination of random variables $C_{w_a, n, 2k}$ and $C_{w_b, n, 2k}$ which are not independent. Hence, the local central limit theorem for the odd $k$ case doesn't immediately follow, but we still conjecture that $G_{n, 2k}$ is asymptotically normally distributed for $k$ odd as well.

\subsection{Number of cone points} Let $s_{n,2k}: \prod^k S_n \rightarrow \NN$ be the random variable counting the number of cone points of a random $S \in \PTS_{n,2k}$. As a Corollary to Theorem \ref{thm:momentsGEN}, we also get the expected value of $s_{n,2k}$.

\begin{corollary} The expected number of cone points on a random polygon-tiled surface is given by,
$$\begin{array}{c|c|c|c|c}
 & k=2 & \text{Even } k\geq 4 & k=3 & \text{Odd }k \geq 5\\[8pt]
\EE[s_{n,2k}] & \log(n) + \gamma -1 + o(1) &  \log(n) + \gamma + o(1) &  2\log(n) + 2\gamma -2 + o(1) &2\log(n) + 2\gamma + o(1) 
\end{array}$$
\end{corollary}
\begin{proof}
Recall that every $S \in \PTS_{n,2k}$ branch covers $S' \in \PTS_{1, 2k}$ where $S'$ is obtained by taking a regular $2k$-gon (oriented in the plane so that one side is vertical) and identifying its opposite sides. There are two branch points (when $k$ is odd) and one branch point (when $k$ is even). 

When $k \geq 4$, the branch points are cone points of $S'$, and hence all their pre-images (i.e. ramification points) on $S$ are cone points as well. Moreover, any cone point on $S$ has to project to a cone point on $S'$. Hence, the random variables $C_{w_c,n,2k}$ (for even $k\geq 4$) and $C_{w_a,n,2k}+C_{w_b,n,2k}$ (for odd $k\geq5$) exactly count the number of cone points on $S$ and subsequently Theorem \ref{thm:moments} implies,
$$ \EE[s_{n,2k}] = \begin{cases}\EE[C_{w_c,n,2k}] = \log(n) + \gamma + o(1) \text{ for even } k \geq 4\\
\EE[C_{w_a,n,2k}] + \EE[C_{w_b,n,2k}] = 2\log(n) + 2\gamma + o(1) \text{ for odd } k \geq 5
\end{cases}$$

 The case $k = 2$ is handled in Corollary \ref{cor:conepoints}. When $k = 3$, $S'$ is the hexagonal torus which does not have cone points. Hence, the pre-images of the branch points may include non-cone points as well. By Proposition \ref{prop:TopFromCommutator}, the number of ramification points on a hexagon-tiled surface $S(\sigma_1, \sigma_2, \sigma_3)$ that are not cone points is exactly the number of fixed points of $w_a(\sigma_1,\sigma_2,\sigma_3)$ together with the fixed points of $w_b(\sigma_1, \sigma_2, \sigma_3)$.
 
So, define $F_{w_i, n, 2k}: \prod^kS_n \rightarrow \ZZ_{\geq 0}$ to count the number of fixed points of $w_i(\sigma_1, \dots, \sigma_k)$ for $i = a, b$. Using Proposition \ref{prop:derangementAn}, for a hexagon-tiled surface,
$$ \EE (s_{n,6}) = \EE(C_{w_a,n,6} + C_{w_b,n,6} - F_{w_a, n, 6} - F_{w_b, n, 6}) = 2\EE(C_{\pi, n}) - 2\EE(F_\pi) = 2 \log(n) + 2\gamma -2 + o(1)$$\end{proof}

\newpage
\bibliographystyle{abbrv}
\bibliography{references}

\begin{thebibliography}{10}

\bibitem{Aul}
D.~Aulicino.
\newblock A new approach to the automorphism group of a platonic surface.
\newblock {\em Rocky Mountain J. Math.}, 50(1):9--23, 2020.

\bibitem{Ben}
E.~A. Bender.
\newblock Central and local limit theorems applied to asymptotic enumeration.
\newblock {\em J. Combinatorial Theory Ser. A}, 15:91--111, 1973.

\bibitem{BrooksMak}
R.~Brooks and E.~Makover.
\newblock Random construction of {R}iemann surfaces.
\newblock {\em J. Differential Geom.}, 68(1):121--157, 2004.

\bibitem{Can}
E.~R. Canfield.
\newblock Application of the {B}erry-{E}ss\'{e}en inequality to combinatorial
  estimates.
\newblock {\em J. Combin. Theory Ser. A}, 28(1):17--25, 1980.

\bibitem{ChmuPit}
S.~Chmutov and B.~Pittel.
\newblock On a surface formed by randomly gluing together polygonal discs.
\newblock {\em Adv. in Appl. Math.}, 73:23--42, 2016.

\bibitem{DiaShah}
P.~Diaconis and M.~Shahshahani.
\newblock Generating a random permutation with random transpositions.
\newblock {\em Z. Wahrsch. Verw. Gebiete}, 57(2):159--179, 1981.

\bibitem{Dix}
J.~D. Dixon.
\newblock The probability of generating the symmetric group.
\newblock {\em Math. Z.}, 110:199--205, 1969.

\bibitem{EskMas}
A.~Eskin and H.~Masur.
\newblock Asymptotic formulas on flat surfaces.
\newblock {\em Ergodic Theory Dynam. Systems}, 21(2):443--478, 2001.

\bibitem{EskMasSchmoll}
A.~Eskin, H.~Masur, and M.~Schmoll.
\newblock Billiards in rectangles with barriers.
\newblock {\em Duke Math. J.}, 118(3):427--463, 06 2003.

\bibitem{EskOk}
A.~Eskin and A.~Okounkov.
\newblock Asymptotics of numbers of branched coverings of a torus and volumes
  of moduli spaces of holomorphic differentials.
\newblock {\em Invent. Math.}, 145(1):59--103, 2001.

\bibitem{FlemPip}
K.~Fleming and N.~Pippenger.
\newblock Large deviations and moments for the {E}uler characteristic of a
  random surface.
\newblock {\em Random Structures Algorithms}, 37(4):465--476, 2010.

\bibitem{ForTangTao}
M.~Forester, R.~Tang, and J.~Tao.
\newblock Veech surfaces and simple closed curves.
\newblock {\em Israel J. Math.}, 223(1):323--342, 2018.

\bibitem{Frob}
G.~Frobenius.
\newblock {\em {\"U}ber gruppencharaktere}.
\newblock Reichsdr., 1896.

\bibitem{Gam}
A.~Gamburd.
\newblock Poisson-{D}irichlet distribution for random {B}elyi surfaces.
\newblock {\em Ann. Probab.}, 34(5):1827--1848, 2006.

\bibitem{GamMak}
A.~Gamburd and E.~Makover.
\newblock On the genus of a random {R}iemann surface.
\newblock In {\em Complex manifolds and hyperbolic geometry ({G}uanajuato,
  2001)}, volume 311 of {\em Contemp. Math.}, pages 133--140. Amer. Math. Soc.,
  Providence, RI, 2002.

\bibitem{GuthParYoung}
L.~Guth, H.~Parlier, and R.~Young.
\newblock Pants decompositions of random surfaces.
\newblock {\em Geom. Funct. Anal.}, 21(5):1069--1090, 2011.

\bibitem{Kol}
V.~F. Kolchin.
\newblock {\em Random graphs}, volume~53 of {\em Encyclopedia of Mathematics
  and its Applications}.
\newblock Cambridge University Press, Cambridge, 1999.

\bibitem{Leb}
J.~L. Lebowitz, B.~Pittel, D.~Ruelle, and E.~R. Speer.
\newblock Central limit theorems, {L}ee-{Y}ang zeros, and graph-counting
  polynomials.
\newblock {\em J. Combin. Theory Ser. A}, 141:147--183, 2016.

\bibitem{Lech}
S.~Lechner.
\newblock Die verteilung des geschlechts zufällig gewählter origamis
  (distribution of the genus of random origamis), 2011.

\bibitem{LelRoy}
S.~Leli\`evre and E.~Royer.
\newblock Orbit countings in $\mathcal{H}(2)$ and quasimodular forms.
\newblock {\em International Mathematics Research Notices}, 2006, January 2006.

\bibitem{LieSha}
M.~W. Liebeck and A.~Shalev.
\newblock Fuchsian groups, coverings of {R}iemann surfaces, subgroup growth,
  random quotients and random walks.
\newblock {\em J. Algebra}, 276(2):552--601, 2004.

\bibitem{Mas2}
H.~Masur.
\newblock Lower bounds for the number of saddle connections and closed
  trajectories of a quadratic differential.
\newblock In {\em Holomorphic functions and moduli, {V}ol. {I} ({B}erkeley,
  {CA}, 1986)}, volume~10 of {\em Math. Sci. Res. Inst. Publ.}, pages 215--228.
  Springer, New York, 1988.

\bibitem{Mas3}
H.~Masur.
\newblock The growth rate of trajectories of a quadratic differential.
\newblock {\em Ergodic Theory Dynam. Systems}, 10(1):151--176, 1990.

\bibitem{MasRafRan}
H.~Masur, K.~Rafi, and A.~Randecker.
\newblock The shape of a generic translation surface, 2018.

\bibitem{Men}
K.~V. Menon.
\newblock On the convolution of logarithmically concave sequences.
\newblock {\em Proc. Amer. Math. Soc.}, 23:439--441, 1969.

\bibitem{MirPet}
M.~Mirzakhani and B.~Petri.
\newblock Lengths of closed geodesics on random surfaces of large genus.
\newblock {\em Comment. Math. Helv.}, 94(4):869--889, 2019.

\bibitem{Nic}
A.~Nica.
\newblock On the number of cycles of given length of a free word in several
  random permutations.
\newblock {\em Random Structures Algorithms}, 5(5):703--730, 1994.

\bibitem{Ore}
O.~Ore.
\newblock Some remarks on commutators.
\newblock {\em Proc. Amer. Math. Soc.}, 2:307--314, 1951.

\bibitem{ParSchul}
O.~Parzanchevski and G.~Schul.
\newblock On the {F}ourier expansion of word maps.
\newblock {\em Bull. Lond. Math. Soc.}, 46(1):91--102, 2014.

\bibitem{Pet2}
B.~Petri.
\newblock Finite length spectra of random surfaces and their dependence on
  genus.
\newblock {\em J. Topol. Anal.}, 9(4):649--688, 2017.

\bibitem{Pet1}
B.~Petri.
\newblock Random regular graphs and the systole of a random surface.
\newblock {\em J. Topol.}, 10(1):211--267, 2017.

\bibitem{PetTha}
B.~Petri and C.~Th\"{a}le.
\newblock Poisson approximation of the length spectrum of random surfaces.
\newblock {\em Indiana Univ. Math. J.}, 67(3):1115--1141, 2018.

\bibitem{PipSchleich}
N.~Pippenger and K.~Schleich.
\newblock Topological characteristics of random triangulated surfaces.
\newblock {\em Random Structures Algorithms}, 28(3):247--288, 2006.

\bibitem{Sach}
V.~N. Sachkov.
\newblock {\em Probabilistic methods in combinatorial analysis}, volume~56 of
  {\em Encyclopedia of Mathematics and its Applications}.
\newblock Cambridge University Press, Cambridge, 1997.
\newblock Translated from the Russian, Revised by the author.

\bibitem{ShresWang}
S.~T. Shrestha and J.~Wang.
\newblock Statistics of square-tiled surfaces: Symmetry and short loops, 2019.

\bibitem{OEIS}
N.~J.~A. Sloane.
\newblock The on-line encyclopedia of integer sequences.
\newblock https://oeis.org/A002104.
\newblock Logarithmic numbers.

\bibitem{SmiWeiss}
J.~Smillie and B.~Weiss.
\newblock Characterizations of lattice surfaces.
\newblock {\em Invent. Math.}, 180(3):535--557, 2010.

\bibitem{Stan1}
R.~P. Stanley.
\newblock {\em Enumerative combinatorics. {V}ol. 1}, volume~49 of {\em
  Cambridge Studies in Advanced Mathematics}.
\newblock Cambridge University Press, Cambridge, 1997.
\newblock With a foreword by Gian-Carlo Rota, Corrected reprint of the 1986
  original.

\bibitem{Stan2}
R.~P. Stanley.
\newblock {\em Enumerative combinatorics. {V}ol. 2}, volume~62 of {\em
  Cambridge Studies in Advanced Mathematics}.
\newblock Cambridge University Press, Cambridge, 1999.
\newblock With a foreword by Gian-Carlo Rota and appendix 1 by Sergey Fomin.

\bibitem{Tam}
T.~Tambour.
\newblock The number of solutions of some equations in finite groups and a new
  proof of {I}t\^{o}'s theorem.
\newblock {\em Comm. Algebra}, 28(11):5353--5361, 2000.

\bibitem{Veech3}
W.~A. Veech.
\newblock Teichm\"{u}ller curves in moduli space, {E}isenstein series and an
  application to triangular billiards.
\newblock {\em Invent. Math.}, 97(3):553--583, 1989.

\bibitem{Zor2}
A.~Zorich.
\newblock Square tiled surfaces and {T}eichm\"{u}ller volumes of the moduli
  spaces of abelian differentials.
\newblock In {\em Rigidity in dynamics and geometry ({C}ambridge, 2000)}, pages
  459--471. Springer, Berlin, 2002.

\end{thebibliography}
\newpage
\appendix
\section{Tools from Combinatorics}\label{sec:combinatorics}
Since our randomizing model is combinatorial, the overall strategy is to convert geometric questions into combinatorial ones about the model, and use combinatorial tools to answer them. Hence, in this section we provide a brief exposition on the relevant combinatorial notions and state the relevant combinatorial lemmas. Majority of the material in this section can be found in the textbook by Stanley \cite{Stan1, Stan2}.

\subsection{Partitions and the hook-length formula} A \textbf{partition} $n \in \NN$ is a sequence $\lambda=(\lambda_1, \dots, \lambda_k) \in \NN^k$ such that $\sum \lambda_i = n$ and $\lambda_1 \geq \lambda_2 \geq \dots \geq \lambda_k$. Two partitions are identical if they only differ in the number of zeros. For example $(3, 2, 1, 1) = (3, 2, 1, 1, 0, 0)$. Informally a partition can be thought of as a way of writing $n$ as a sum $\lambda_1 + \dots + \lambda_k$ disregarding the order the $\lambda_i$ (since there is a canonical way of writing such a sum as a partition). If $\lambda$ is a partition of $n$, we write $\lambda \vdash n$. The non-zero $\lambda_i$ are called parts of $\lambda$ and we say that $\lambda$ has $k$ parts if $k = \#\{i: \lambda_i > 0\}$. If $\lambda$ has $\xi_i$ parts equal to $i$, then we can write $\lambda = \ideal{1^{\xi_1}, 2^{\xi_2}, \dots}$ where terms with $\xi_i = 0$ and the superscript $\xi_i = 1$ is omitted. Denote by $\Par(n)$, the set of all partitions of $n$ and set $\Par = \bigcup_{n \in \NN} \Par(n)$.

%

For a partition $(\lambda_1, \dots, \lambda_k) \vdash n$, we can draw a left-justified array of boxes with $\lambda_i$ boxes in the $i$-th row. This is called the \textbf{Young diagram} associated to partition $(\lambda_1, \dots, \lambda_k)$. The squares in a Young diagram can be identified using tuples $(i,j)$ where $i$ is the row corresponding to part $\lambda_i$ and $1 \leq j\leq \lambda_i$ is the position of the square along that row. Given a square $r = (i,j) \in \lambda$, define the \textbf{hook length of $\lambda$ at $r$} as the number of squares directly to the right or directly below $r$, counting $r$ itself once. The \textbf{hook length product of $\lambda$}, denoted $H_\lambda$ is the product, 
$$H_\lambda = \prod_{u \in \lambda} h(u)$$
Likewise, define the \textbf{content} $\cont(r)$ of $\lambda$ at $r = (i, j)$ by $\cont(r) = j-i$. In general we obtain a \textbf{Young tableau} by filling in the boxes of the Young diagram with entries from a totally ordered set (usually a set of positive integers). It is called a \textbf{semistandard Young tableau} (SSYT) if the entries of the tableau weakly increase along each row and strictly increase down each column. The type of an SSYT is a sequence $\xi = (\xi_1, \dots )$ where the SSYT contains $\xi_1 1's$, $\xi_2 2's$ and so on. See Figure \ref{fig:youngdiagram} for examples of Young diagram and Young tableau.

\begin{figure}[h!!]
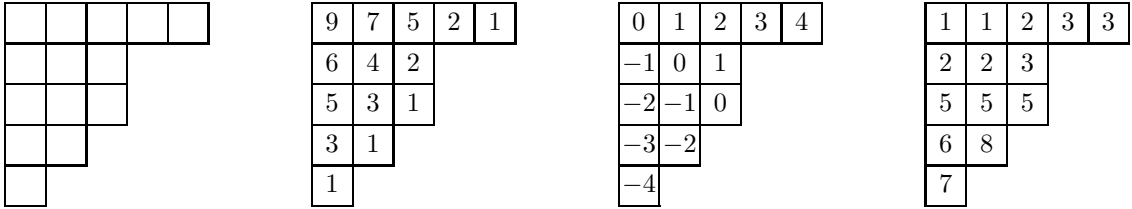

\begin{minipage}{0.24\linewidth}
\centering
$\begin{ytableau} {}&&&& \cr {}&& \cr {}&& \cr {}& \cr {}
\end{ytableau}$
\end{minipage}
\begin{minipage}{0.24\linewidth}
\centering
$\begin{ytableau} {9}&{7}&{5}&{2}&{1}\cr {6}&{4}&{2} \cr {5}&{3}&{1} \cr {3}&{1} \cr {1}
\end{ytableau}$
\end{minipage}
\begin{minipage}{0.24\linewidth}
\centering
$\begin{ytableau} {0}&{1}&{2}&{3}&{4} \cr {-1}&{0}&{1} \cr {-2}&{-1}&{0} \cr {-3}&{-2} \cr {-4}
\end{ytableau}$
\end{minipage}
\begin{minipage}{0.24\linewidth}
\centering
$\begin{ytableau} {1}&{1}&{2}&{3}&{3} \cr {2}&{2}&{3} \cr {5}&{5}&{5} \cr {6}&{8} \cr {7}
\end{ytableau}$
\end{minipage}
\caption{Young diagram and tableau for the partition (5,3,3,2,1). From left to right: Young diagram;  Tableau with hook lengths; Tableau with contents; Example of an SSYT}
\label{fig:youngdiagram}
\end{figure}

Partitions are particularly relevant to us, since the cycle type of a permutation $\pi \in S_n$ is a partition of $n$, and two permutations are conjugate if and only if they have the same cycle type. Hence, partitions index conjugacy classes of $S_n$, and we will denote $\cyc(\pi)$ to be the partition obtained from the cycle type of $\pi$. Moreover, partitions of $n$ also index the irreducible representations and irreducible characters of $S_n$ canonically. We will denote by $\rho^\lambda$ and $\chi^\lambda$ the irreducible representation and character associated indexed by the partition $\lambda$. One instance of the deep relation between partitions and the representation theory of the symmetric group is the Hook-length formula which gives a relation between the degree (or dimension)  of a representation and the hook length product of the partition that is used to index it:

\begin{lemma}[Hook-length formula]\label{lem:hooklength}
For a partition $\lambda \vdash n$, and associated representation $\rho^\lambda \in \hat{S}_n$, the degree is given by $$ \dim(\rho^\lambda) = \frac{n!}{H_\lambda}$$
\end{lemma}

\subsection{Symmetric functions and the Murnaghan-Nakayama Rule}

Let $x = (x_1, x_2,\dots)$ be a set of indeterminates, and let $n \in \NN$. 

Denote the set of all homogeneous symmetric functions of degree $n$ over $\QQ$ as $\Lambda^n$. Note that $\Lambda^n$ is a $\QQ$-vector space with many different standard bases, two of which we will utilize. The first of which are called \textbf{power sum symmetric functions}, denoted $p_\lambda$ and indexed by partitions $\lambda$. They are defined as, 
\begin{align*}
&p_n := \sum_{i} x_i^n, \hspace{1cm} n \geq 1 \hspace{1cm} \text{ (with }p_0 = 1)\\
&p_\lambda := p_{\lambda_1}p_{\lambda_2}\dots \hspace{1cm} \text{ if }\lambda = (\lambda_1, \lambda_2, \dots, )
\end{align*}
The second set of symmetric functions we will need are called \textbf{Schur functions}, denoted $s_\lambda$ and also indexed by partitions $\lambda$. They are defined as the formal power series,
$$ s_\lambda(x) = \sum_{T} x^T$$
where
\begin{itemize}
\item the sum is over all SSYTs $T$ of shape $\lambda$. 
\item $x^T = x_1^{\xi_1(T)} x_2^{\xi_2(T)} \dots$ if $T$ is an SSYT of type $\xi$.
\end{itemize}
While it is not hard to see that power sum symmetric functions are indeed symmetric functions, it is a non-trivial theorem that Schur functions are indeed symmetric functions. There is a deep connection between Schur functions, power sum symmetric functions and irreducible characters of $S_n$. This relation is called the Murnaghan-Nakayama rule:

\begin{theorem}[Murnaghan-Nakayama Rule] For $\lambda \vdash n$,
$$ s_\lambda = \sum_{\nu \vdash n} z^{-1}_\nu \chi^\lambda (\nu)p_\nu$$
where $\chi^\lambda$ is the irreducible character of $S_n$ indexed by $\lambda$ and $z^{-1}_\nu = \frac{\#\{\sigma \in S_n: \cyc(\sigma) = \nu\}}{n!}$
\end{theorem}

Conversely, the power sum symmetric functions can be transformed into the Schur functions as follows:

\begin{theorem}\label{thm:powersumtoschur} The power sum symmetric functions can be expressed as a linear combination of the Schur functions in the following way:
$$ p_\mu = \sum_{\lambda \vdash n} \chi^\lambda(\mu) s_\lambda$$
\end{theorem}

Next, let $\CF^n$ be the set of class functions $f: S_n \rightarrow \QQ$. Then, there exists a natural inner product on $\CF^n$ given by,
$$ \ideal{f, g} = \frac{1}{n!}\sum_{\sigma \in S_n} f(\sigma)g(\sigma) = \sum_{\lambda\vdash n} z^{-1}_\lambda f(\lambda)g(\lambda)$$
where $f(\lambda)$ denotes the value of $f$ on the conjugacy class associated to partition $\lambda$. To each class function $f \in \CF^n$, one can associate a symmetric function of degree $n$ via the linear transformation $\ch: \CF^n \rightarrow \Lambda^n$, called the \textbf{Frobenius characteristic map} given by,

\begin{equation}\label{eq:frobeniuscharmap1} \ch f = \frac{1}{n!}\sum_{\sigma \in S_n} f(\sigma)p_{\cyc(\sigma)} = \sum_{\lambda \vdash n} z^{-1}_\lambda f(\lambda) p_\lambda.
\end{equation} Using Theorem \ref{thm:powersumtoschur}, the characteristic function is expressed as,
\begin{equation}\label{eq:frobeniuscharmap2} \ch f = \sum_{\lambda \vdash n} \ideal{f, \chi^\lambda}s_\lambda
\end{equation}

Finally, we state a particular specialization of $s_\lambda$ which we will use. This is Corollary 7.21.4 of \cite{Stan2}

\begin{lemma}\label{lem:schurspec} For any $\lambda \in \Par$ and $m$ a positive integer, we have

$$s_\lambda(1^m) = \prod_{r \in \lambda} \frac{m+\cont(r)}{h(r)}$$ 
where $s_\lambda(1^m)$ means evaluating $s_\lambda$ by setting $x_1 = \dots x_m = 1$ and $x_i = 0$ for all $i > m$.
\end{lemma}


\end{document}